\documentclass[runningheads]{llncs}

\usepackage[utf8]{inputenc}
\usepackage{graphicx}
\usepackage{subcaption}
\usepackage{amsmath,amssymb}
\usepackage{booktabs}
\usepackage[inline]{enumitem}
\usepackage{color}
\usepackage[basic]{complexity}

\usepackage{footnote}
\makesavenoteenv{tabular}
\usepackage[symbol]{footmisc}

\usepackage[nocompress,noadjust]{cite}
\usepackage[dvipsnames,usenames,cmyk]{xcolor}
\usepackage[colorlinks=true, urlcolor=blue, citecolor=green, linkcolor=blue, breaklinks, unicode]{hyperref}
\usepackage[nameinlink,capitalize]{cleveref}

\usepackage[colorinlistoftodos,prependcaption,textsize=tiny]{todonotes}

\definecolor{mypurple}{rgb}{0.627, 0.125, 0.941}
\definecolor{myblue}{rgb}{0, 0, 1}

\newtheorem{obs}{Observation}
\newtheorem{open}{Open Problem}
\newenvironment{sketch}{\medskip\noindent\emph{Proof (sketch).}}{\mbox{}\hfill$\qed$\par\medskip}

\graphicspath{{pics/}}

\newcommand{\df}[1]{{\it #1}}
\newcommand{\Oh}{{\ensuremath{\mathcal{O}}}}
\newcommand{\Sh}{{\ensuremath{\mathcal{S}}}}
\newcommand{\tw}[1]{\ensuremath{twist(#1)}}
\newcommand{\arr}{\ensuremath{\!\rightarrow\!}}

\newcounter{dummycount}
\newcommand{\wormholeThm}[1]{
	\newcounter{#1}
	\setcounter{#1}{\value{theorem}}}
\newenvironment{backInTimeThm}[1]{
	\setcounter{dummycount}{\value{theorem}}
	\setcounter{theorem}{\value{#1}}}
{\setcounter{theorem}{\value{dummycount}}}
\newcommand{\wormholeLm}[1]{
	\newcounter{#1}
	\setcounter{#1}{\value{lemma}}}

\newenvironment{backInTimeLm}[1]{
	\setcounter{dummycount}{\value{lemma}}
	\setcounter{lemma}{\value{#1}}}
{\setcounter{lemma}{\value{dummycount}}}

\renewcommand{\orcidID}[1]{\href{https://orcid.org/#1}{\includegraphics[scale=.03]{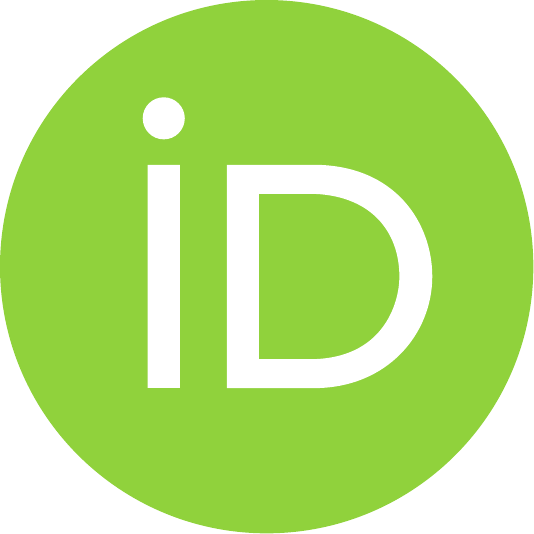}}} 

\begin{document}
	
\title{On Families of Planar DAGs with\\ Constant Stack Number}

\author{
	Martin~N{\"o}llenburg\inst{1}\orcidID{0000-0003-0454-3937} \and
	Sergey~Pupyrev\inst{2}\orcidID{0000-0003-4089-673X}}

\authorrunning{M. N{\"o}llenburg and S. Pupyrev}
\institute{
	Algorithms and Complexity Group, TU Wien, Austria
	\email{noellenburg@ac.tuwien.ac.at} \and
	Meta, Menlo Park, CA, USA\\
	\email{spupyrev@gmail.com}}

\maketitle              %
	
\begin{abstract}
A $k$-stack layout (or $k$-page book embedding) of a graph consists of a total order of the vertices, 
and a partition of the edges into $k$ sets of non-crossing edges with respect to the vertex order.
The stack number of a
graph is the minimum $k$ such that it admits a $k$-stack layout.

In this paper we study a long-standing problem regarding the stack number of planar directed acyclic graphs (DAGs),
for which the vertex order has to respect the orientation of the edges.
We investigate upper and lower bounds on the stack number of several families of planar graphs: We improve the constant upper bounds on the stack number of single-source and monotone outerplanar DAGs and of outerpath DAGs, and improve the constant upper bound for upward planar 3-trees. Further, we provide computer-aided lower bounds for upward (outer-) planar DAGs.

\end{abstract}

\section{Introduction}

Let $G=(V, E)$ be a simple graph with $n$ vertices and $\sigma$ be a total order of the vertex set $V$. 
Two edges $(u, v)$ and $(w, z)$ in $E$ with $u <_{\sigma} w$ \df{cross} if $u <_{\sigma} w <_{\sigma} v <_{\sigma} z$. 
A \df{$k$-stack layout} (\df{$k$-page book embedding}) of $G$ is a total order of $V$ and a partition of $E$ into 
$k$ subsets,
called \df{stacks} or \df{pages}, such that no two edges in the same subset cross. The \df{stack number} 
(\df{page number}, \df{book thickness})
of $G$ is the minimum $k$ such that $G$ admits a $k$-stack layout.

Heath et al.~\cite{HPT99a,HPT99b} extended the notion of stack
number to directed acyclic graphs (DAGs for short) in a natural way:
Given a DAG, $G=(V, E)$, a book embedding of $G$ is defined as
for undirected graphs, except that the total order $\sigma$ of $V$ is now required to be
a \df{linear extension} of the partial order of $V$ induced by $E$. That is, if $G$ contains
a directed edge $(u,v)$ from a vertex $u$ to a vertex $v$, then $u <_{\sigma} v$ in any feasible total order $\sigma$
of $V$. Heath et al. showed that DAGs with stack number
1 can be characterized and recognized efficiently; however, they proved that,
in general, determining the stack number of a DAG is \NP-complete.

The main problem raised by Heath et al.~\cite{HPT99a,HPT99b} and studied in several papers~\cite{AR96,HP97,GGLW06,FFR13,BLGDMP19}
is whether every upward planar DAG has constant stack number.
Recall that an \df{upward planar} DAG is a DAG that admits a 
drawing which is simultaneously \df{upward}, that is, each edge is represented by a curve monotonically 
increasing in the y-direction, and \df{planar}, that is, no two edges cross each other.

\begin{open}
	\label{open:planar}
	Is the stack number of every upward planar DAG bounded by a constant?
\end{open}

Notice that upward planarity is a necessary condition for the question: 
there exist DAGs which admit a planar non-upward embedding and that require $\Omega(n)$ stacks in any 
book embedding~\cite{HPT99a}; see \cref{fig:ns}.

\begin{figure}[!tb]
	\begin{subfigure}[b]{.45\linewidth}
		\center
		\includegraphics[page=5]{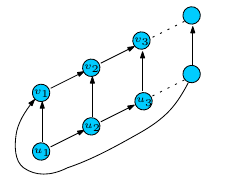}
		\caption{A DAG that requires $n/2$ stacks: edges $(u_i, v_i)$ form an $n/2$-twist}
		\label{fig:ns}
	\end{subfigure}    
	\hfill
	\begin{subfigure}[b]{.45\linewidth}
		\center
		\includegraphics[page=6,height=80pt]{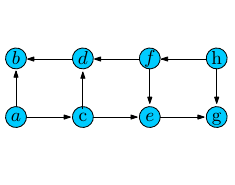}
		\caption{An outerplanar DAG which is not upward planar~\cite{Pap94}}
		\label{fig:no-odag}
	\end{subfigure}    
	\caption{Planar DAGs that (a) need many stacks or (b) are not upward.}
\end{figure}

In its general form, \cref{open:planar} is still unresolved.
Heath et al.~\cite{HPT99a,HPT99b} showed that directed trees and unicyclic DAGs have stack 
numbers $1$ and $2$, respectively. Mchedlidze and Symvonis~\cite{MS09} proved that 
N-free upward planar DAGs, which contain series-parallel digraphs, have stack number $2$.
Frati et~al.~\cite{FFR13} gave several conditions under which upward planar triangulations 
have bounded stack number. In particular, they showed that (i) maximal upward planar $3$-trees have a 
constant stack number, and (ii) planar triangulations with a bounded (directed) diameter have a constant
stack number.
Notice that the graph in \cref{fig:ns}, that requires $\Omega(n)$ stacks, is a partial planar $3$-tree. 
Thus, it is reasonable to ask whether the stack number is bounded for (non-upward but directed acyclic) $2$-trees or
their subfamilies, outerplanar graphs, also known as simple $2$-trees.
This question has been first asked by Heath et al.~\cite{HPT99a} and 
recently highlighted by Bekos et al.~\cite{Bek19}.\footnote{Very recently, Jungeblut et al.~\cite{jmu-daoghcsn-22} resolved the problem by proving that every outerplanar DAG has constant stack number, upper bounded by $24776$, while there are directed acyclic (non upward planar) 2-trees with unbounded stack number. Their proof of the upper bound relies on Theorem 1b (see below) as a central tool and their second result solves an open question raised in a preprint version of this paper.}
Bhore et al.~\cite{bdmn-ubtpccr-22} gave upper bounds for some upward outerplanar graphs, namely internally-triangulated outerpaths ($16$ stacks), cacti ($6$ stacks), and upward outerplanar graphs whose biconnected components are $st$-outerplanar ($8$ stacks).

We emphasize that directed acyclic $2$-trees are planar but not necessarily upward, and thus, 
the results of Frati et al.~\cite{FFR13} do not apply for this class of graphs. For example, the graph
in \cref{fig:no-odag} is a directed acyclic partial $2$-tree (in fact, it is an outerpath DAG) but it 
cannot be drawn in an upward fashion.

\paragraph{Our Contributions.}
We investigate upper and lower bounds for the stack number of upward planar DAGs and outerplanar DAGs (oDAGs for short).
Throughout the paper, we express the bounds in terms of the maximum size of
a \df{twist} in the vertex order, that is, the maximum number of mutually crossing edges.
This parameter, also called the \df{twist number} of a graph, is tied to the stack number; 
analyzing the maximum twist size significantly simplifies the arguments at the cost of (slightly) worsened bounds
for the stack number.
We refer to \cref{sect:prel} for details and formal definitions.

In \cref{sect:odag}, we present constant upper bounds for several prominent subclasses of outerplanar DAGs. 

\begin{theorem}
	\label{thm:odag}
	~
	\begin{enumerate}[label=\bf\emph{\alph*}.,topsep=0pt,partopsep=0pt,itemsep=0pt]
		\item Every single-source outerplanar DAG has a constant stack number with a vertex order whose twist size is at most $3$.
		\item Every monotone outerplanar DAG has a constant stack number with a vertex order whose twist size is at most $4$.
		\item Every outerpath DAG has a constant stack number with a vertex order whose twist size is at most $4$.
	\end{enumerate} 
\end{theorem} 

The recent result of Davies~\cite{d-ibccg-22} implies that every graph with a vertex order whose twist size is at most $k$, 
has stack number at most $2k \log_2 k + 2k \log_2\log_2 k + 10k$ (and for $k=3$ Davies proves an upper bound of $19$).
It follows that single-source oDAGs have stack number at most $19$, while monotone oDAGs and outerpath DAGs have stack number 
at most $64$. We note that the stack assignment for the provided vertex orders can likely be improved. For example, 
we show an upper bound of $4$ stacks for single-source oDAGs (refer to \cref{thm:ss_odag_stack} in \cref{sect:odag}). 

Our proof technique utilized for \cref{thm:odag} can be applied to other classes of DAGs.
In \cref{sect:up3t} we tighten the upper bound on the stack number of upward (maximal) planar $3$-trees.
Frati et al.~\cite{FFR13} bound the stack number of upward planar $3$-trees by a function of
the size of the maximum twist size without providing an explicit bound.
We strengthen their results by presenting an arguably simpler proof that yields an exact (small) bound of $5$ on the maximum twist size (which by Davies' result~\cite{d-ibccg-22} translates into a stack number of at most $85$).

\wormholeThm{thm-up3t}
\begin{theorem}
	\label{thm:up3t}
	Every upward planar $3$-tree has a constant stack number with a vertex order whose twist size is at most $5$.
\end{theorem} 

The proofs of \cref{thm:odag} and \cref{thm:up3t} are constructive and lead to linear-time algorithms for 
constructing the vertex orders.

Finally, we explore lower bounds on the stack number of planar DAGs in \cref{sect:lower}.
They rely on computational experiments using a SAT formulation of the book embedding problem. %

\wormholeThm{thm-other}
\begin{theorem}
	\label{thm:other} 
	~
	\begin{enumerate}[label=\bf\emph{\alph*}.,topsep=0pt,partopsep=0pt,itemsep=0pt]
		\item There exists a single-source single-sink upward outerplanar DAG with stack number $3$.
		\item There exists an upward outerplanar DAG with stack number $4$.
		\item There exists an upward planar 3-tree DAG with stack number $5$.
	\end{enumerate}
\end{theorem}

\paragraph{Other Related Work.}

Book embeddings of undirected graphs received a lot of attention due to their numerous applications.
It is known that the graphs with stack number $1$ are exactly outerplanar graphs, while graphs with
stack number $2$ are exactly the subhamiltonian graphs, %
which implies
that it is \NP-complete to decide whether a graph admits a $2$-stack layout. More generally, 
every planar graph has stack number at most $4$, and the bound is worst-case optimal~\cite{Yan89,BKKPRU20}.

Stack numbers of directed acyclic graphs have also been extensively studied. 
Similarly to the undirected case, it is \NP-complete to test whether the stack
number of a DAG is at most $k$, even when $k=2$~\cite{blfgmr-rdwpn-23}.
Several works analyzed the stack number of partially ordered sets (posets), which
can be viewed as upward planar DAGs without \df{transitive} edges.
Nowakowski and Parker~\cite{NP89} asked whether the stack number of a planar poset is bounded by a constant.
Notice that the question is a special case of \cref{open:planar}.
Several works provide bounds for the stack number of special classes of posets and
bounds in terms of various parameters (e.g., height or bump number)~\cite{AJZ15,H89,Sys89}.
To  our knowledge, there is no indication that the absence of transitive edges simplifies
\cref{open:planar}.

As for the lower bounds on the stack number of DAGs and posets, not many results are known.
It is easy to construct a planar poset with  stack number $4$~\cite{H89,AJZ15}, which 
for a long time has been the best known lower bound for the stack number of upward planar DAGs.
Our \cref{thm:other} strengthens the result by showing that there exist (maximal) upward
planar $3$-trees with stack number~$5$. Merker~\cite{M20} independently constructed a planar poset with stack number~$5$.
Jungeblut et al.~\cite{jmu-sbpnupg-22} further showed that upward planar graphs of constant width and height
have a bounded stack number, and combined the two results to get an $\Oh(n^{2/3} \log(n)^{2/3})$ upper
bound on the stack number of general upward planar graphs. 
Yet these results do not
imply any upper bound for the graph classes considered in \cref{sect:odag} (since oDAGs can be non-upward)
or in \cref{sect:up3t} (since upward planar 3-trees can have linear width and height).

\medskip

\noindent \textit{Due to space constraints, details of omitted/sketched proofs are in the appendix.}

\section{Preliminaries}
\label{sect:prel}

Throughout the paper, $G = (V, E)$ is a simple directed graph (digraph) with vertex set $V$ and edge (arc) set $E$.
A \df{vertex order}, $\sigma$, of a digraph $G$ is a linear extension of $V$. That is, 
if $G$ contains an edge from a vertex $u$ to a vertex $v$, denoted $(u, v) \in E$, then
$u <_{\sigma} v$ in any feasible vertex order $\sigma$ of $V$.
Let $F$ be a set of $k \geq 2$ independent (that is, having no common endpoints) edges 
$(s_i, t_i), 1 \le i \le k$. 
If $s_1 <_{\sigma} \dots <_{\sigma} s_k <_{\sigma} t_1 <_{\sigma} \dots <_{\sigma} t_k$, then
$F$ is a \df{$k$-twist}. Two independent edges forming a $2$-twist are 
called \df{crossing}.
A \df{$k$-stack layout} of $G$ is a pair $(\sigma, \{\Sh_1, \dots, \Sh_k\})$, 
where $\sigma$ is a vertex order of $G$ and $\{\Sh_1, \dots, \Sh_k\}$ is a partition of $E$ into 
\df{stacks}, that is, sets of pairwise non-crossing edges.
The minimum number of stacks in a stack layout of $G$ is its \df{stack number}.

The size of the largest twist in a vertex order is tied to the number of stacks needed for
the edges of the graph under the vertex order. In one direction, a vertex order with a $k$-twist needs at least 
$k$ stacks, since each edge of a twist must be in a distinct stack. In the other direction, a vertex order with no
$(k+1)$-twist needs at most $\Oh(k \log k)$ stacks~\cite{d-ibccg-22}, which matches the lower bound of $\Omega(k \log k)$~\cite{KK97}.
An order without a $2$-twist (that is, when $k=1$) corresponds to an outerplanar drawing of a graph, which is
a $1$-stack layout.
For $k=2$ (an order without a $3$-twist), $5$ stacks are sufficient and sometimes necessary~\cite{K88,A96}.

In the following we use notation $E(V_1 \arr V_2)$ to indicate a subset of $E$ between disjoint
subsets $V_1, V_2 \subseteq V$, that is, 
$(x, y) \in E$ for $x \in V_1, y \in V_2$.
Notation $E(V_1 \arr V_2, V_3 \arr V_4, \dots)$ indicates the union of the edge sets, that is, 
$E(V_1 \arr V_2) \cup E(V_3 \arr V_4) \cup \dots$.
Similarly, we write $\tw{V_1 \arr V_2} \le k$ to indicate that
the maximum twist of the edges $E(V_1 \arr V_2)$ is of size at most $k$.
Slightly abusing the notation, we sometimes write $E(v \arr V_1)$ or $E(V_1 \arr v)$,
where $v \in V \setminus V_1$ and $V_1 \subset V$.
To specify a relative order between disjoint subsets of vertices, we use $\sigma = [V_1, V_2, \dots, V_r]$, 
where $V_i \subseteq V$ for $1 \le i \le r$. For the vertex order $\sigma$, it holds that 
$x <_{\sigma} y$ for all $x \in V_i$, $y \in V_j$ such that $i < j$.

\section{Outerplanar DAGs}
\label{sect:odag}
We study the stack number of oDAGs, that is, directed acyclic outerplanar graphs.
We stress that such graphs are planar but not necessarily upward. For example, the graph
in \cref{fig:no-odag} cannot be drawn in an upward fashion. 
We assume oDAGs are maximal as it is straightforward to augment an oDAG to a
maximal one, and the stack number is a monotone parameter
under taking subgraphs.

It is well-known that every maximal outerplanar directed acyclic graph can be constructed from an 
edge, which we call the \df{base} edge, by repeatedly \df{stellating} edges~\cite{jmu-daoghcsn-22}; that is,
picking an edge, $(s, t)$, on its outerface 
and adding a vertex $x$ together with two edges connecting $x$ with $s$ and $t$; see \cref{fig:operations}.
In order to keep the graph acyclic, the directions of the new edges must be either \df{transitive}:

\smallskip
\noindent\begin{enumerate*}[label=\textbf{O.\arabic*},ref=O.\arabic*]
	\item\label{op1} $(s, x) \in E$ and $(x, t) \in E$,
\end{enumerate*} 
or \df{monotone}:

\smallskip
\noindent\begin{enumerate*}[label=\textbf{O.\arabic*},ref=O.\arabic*,resume]
	\item\label{op2} $(s, x) \in E$ and $(t, x) \in E$, \qquad\qquad
	
	\item\label{op3} $(x, s) \in E$ and $(x, t) \in E$.
\end{enumerate*} 

\smallskip

We emphasize that every edge, including the base edge, in the construction sequence of outerplanar graphs 
can be stellated at most once; relaxing the condition yields a construction scheme for $2$-trees.

We study subclasses of outerplanar DAGs that can be constructed using a subset of the three operations.
First observe that so-called \df{transitive} oDAGs that are constructed from an edge
by applying \ref{op1} have a single source vertex, a single sink vertex, and an edge connecting
the source with the sink. Such graphs are trivially embeddable in one stack.
In \cref{sect:ss_odag} we observe that single-source oDAGs can be constructed using
\ref{op1} and \ref{op2}; similarly, single-sink oDAGs can be constructed by \ref{op1} and \ref{op3}.
We show that single-source (single-sink) oDAGs admit a layout in a constant number of stacks.
Furthermore, using monotone operations (\ref{op2} and \ref{op3}), one can construct
outerplanar graphs with arbitrarily many sources and sinks. Such \df{monotone} oDAGs
admit layouts in a constant number of stacks, as we prove in \cref{sect:mon_odag}.
Finally, we investigate \df{outerpath} DAGs, that is, oDAGs whose weak dual is a path.
In \cref{sect:outerpath_dag} we describe a construction scheme for such graphs and prove that their
stack number is constant.

\begin{figure}[!tb]
    \begin{subfigure}[b]{.45\linewidth}
		\center
		\includegraphics[page=4,width=\textwidth]{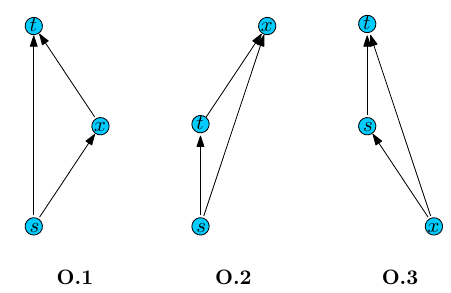}
		\caption{}
		\label{fig:operations}
	\end{subfigure}   	 
	\hfill 
    \begin{subfigure}[b]{.45\linewidth}
		\center
		\includegraphics[page=3,width=\textwidth]{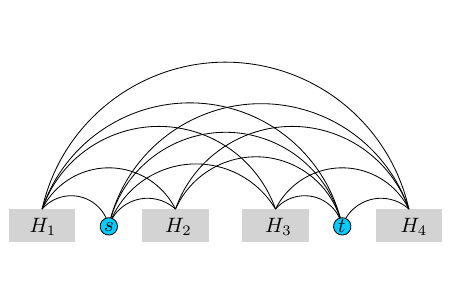}
		\caption{}
		\label{fig:invariants}
	\end{subfigure}   	
 	\caption{(a)~Possible ways of stellating base edge $(s, t)$ with a vertex $x$ for
	constructing oDAGs. (b) Vertex order utilized for the inductive schemes in \cref{sect:odag}.}
\end{figure}

All our proofs are based on an inductive scheme by decomposing a given oDAG into two subgraphs
that can be embedded so that a list of carefully chosen invariants is maintained. Then we show
how to combine the layouts of the subgraphs and verify the invariants. To this end, we consider a base edge
$(s, t) \in E$ and define a vertex order consisting of six vertex-disjoint parts
$\sigma = [H_1$, $s$, $H_2$, $H_3$, $t$, $H_4]$, where $H_i \subset V, 1 \le i \le 4$.
For all the considered graph classes, we require that $E(H_2 \arr H_3) = \emptyset$; see \cref{fig:invariants}.
In all figures in the paper all edges are oriented from left to right unless the arrows explicitly indicate
edge directions.

\subsection{Single-Source oDAGs}
\label{sect:ss_odag}

Here we consider \df{single-source} (\df{single-sink}) outerplanar DAGs that contain only one
source (sink) vertex. Single-source oDAGs can be constructed from an edge by 
applying two of the operations, \ref{op1} and \ref{op2}. To this end, choose an edge incident to the
source on the outerface of the graph as the base edge, and observe that applying \ref{op3} would create a predecessor of the source or an additional source. Similarly, single-sink graphs can be constructed by two operations, \ref{op1} and \ref{op3}.

\wormholeLm{lm-ss-odag-twist}
\begin{lemma}
	\label{thm:ss_odag_twist}
	Every single-source (single-sink) outerplanar DAG admits an order whose twist size is at most $3$.
\end{lemma}

\begin{sketch}
	Let $G = (V, E)$ be a given oDAG with a unique source $s \in V$, and assume that $(s, t) \in E$ is 
	the base edge in the construction sequence of $G$.	
	We prove the claim by induction on the size of $G$ by using
	the following invariant (see \cref{fig:ss_odag_inv}):
	There exists an order of $V$ consisting of four parts, $\sigma = [s , H_3, t, H_4]$ (that is, $H_1 = H_2 = \emptyset$), such that the following holds:

\smallskip
	\noindent~~\begin{enumerate*}[label=\textbf{I.\arabic*},ref=I.\arabic*]
		\item\label{ss:I1} $\tw{s \arr H_4, H_3 \arr H_4} \le 1$ \qquad\qquad
		
		\item\label{ss:I2} $\tw{E} \le 3$
	\end{enumerate*} 
\smallskip

	\begin{figure}[!t]
		\begin{subfigure}[b]{.40\linewidth}
			\center
			\includegraphics[page=1,height=40pt]{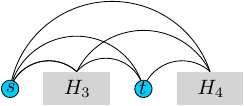}
			\caption{The invariant used in the proof}
			\label{fig:ss_odag_inv}
		\end{subfigure}    
		\hfill
		\begin{subfigure}[b]{.25\linewidth}
			\center
			\includegraphics[page=3,height=100pt]{pics/ss_odag}
			\caption{Case~1}
			\label{fig:ss_odag_case1}
		\end{subfigure}   
		\begin{subfigure}[b]{.25\linewidth}
			\center
			\includegraphics[page=2,height=100pt]{pics/ss_odag}
			\caption{Case~2}
			\label{fig:ss_odag_case2}
		\end{subfigure}   
		\caption{An illustration for \cref{thm:ss_odag_twist}}
	\end{figure}

	Now we prove that these invariants can be maintained.
	If $G$ is a single edge, then the base of the induction 
	clearly holds. For the inductive case, we consider
	the base edge $(s, t)$ of $G$. Let $x$ be the unique neighbor of $s$ and $t$. 
    Since $G$ is a single-source oDAG, there are
	two ways the edges between $x$ and $s, t$ are directed, corresponding to
	operations \ref{op1} and \ref{op2}. Consider both cases.
	
	\paragraph{Case 1.} 
	First assume $(s, x) \in E$ and $(t, x) \in E$. It is easy to see that $G$ is decomposed into two edge-disjoint 
	subgraphs sharing a single vertex $x$; denote the graph containing $(s, x)$ by $G^g$ and the graph containing
	$(t, x)$ by $G^r$; see \cref{fig:ss_odag_case1}. 
	Since $G$ is a single-source oDAG, $G^g$ is also a single-source oDAG with source $s$ and base edge $(s, x)$.
	Similarly, $G^r$ is a single-source oDAG with source $t$ and base edge $(t, x)$.
	By the induction hypothesis, the graphs admit orders $\sigma^g$ and $\sigma^r$ satisfying the
	invariant. Next we combine the orders into a single~one~for~$G$. 
	
	Let $\sigma^g = [s, G_3, x, G_4]$ and $\sigma^r = [t, R_3, x, R_4]$. 
	Then we set 
	\[	
	\sigma = [s, t, R_3, G_3, x, G_4, R_4]
	\]	
	and observe that in the order
	$H_3 = \emptyset$, $H_4 = [R_3, G_3, x, G_4, R_4]$; see \cref{fig:ss_odag_case1_ind}.
	It is easy to see that $\sigma$ is a linear extension of $V$, that is, $u <_{\sigma} v$ for all edges $(u, v) \in E$.
	Next we verify the conditions of the invariant.
	
	\begin{figure}[!ht]
		\begin{subfigure}[b]{.5\linewidth}
			\center
			\includegraphics[page=4,height=90pt]{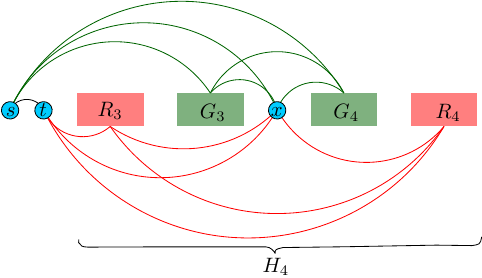}
			\caption{Case~1}
			\label{fig:ss_odag_case1_ind}
		\end{subfigure}   
		\begin{subfigure}[b]{.5\linewidth}
			\center
			\includegraphics[page=5,height=90pt]{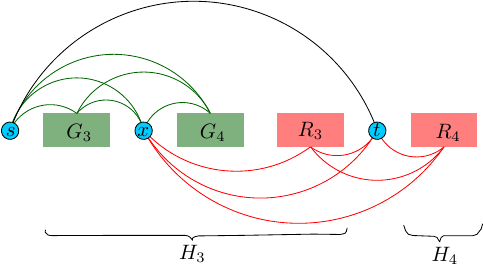}
			\caption{Case~2}
			\label{fig:ss_odag_case2_ind}
		\end{subfigure}   
		\caption{An inductive step in the proof of \cref{thm:ss_odag_twist}}
	\end{figure}  
	
	\begin{enumerate}
		\item[\ref{ss:I1}.] 
		Since $H_3 = \emptyset$, we have 
		$\tw{s \arr H_4, H_3 \arr H_4} = \tw{s \arr H_4} \le 1$, where
		the inequality holds because all edges $E(s \arr H_4)$ share a common vertex, $s$.
		
		\item[\ref{ss:I2}.] 
		Consider the maximum twist $\kappa$ in $G$ under vertex order $\sigma = [s, H_3, t, H_4]$, and
		suppose for contradiction that $|\kappa| \ge 4$.
		Observe that $H_3 = \emptyset$ in the considered case, and $\kappa$ may contain at most one edge incident to $s$
		and at most one edge incident to $t$. Thus, at least two of the edges of $\kappa$
		are from $E(H_4 \arr H_4)$; see \cref{fig:ss_odag_case1_ind}.
		Denote one of the two edges by $e \in \kappa$.
		
		Since $e \in E(H_4 \arr H_4)$, both endpoints of $e$ are in $R_3 \cup G_3 \cup \{x\} \cup G_4 \cup R_4$.
		Notice that if the two endpoints are in the same part (e.g., $R_i$ or $G_i$ for some $i$), then all edges of
		$\kappa$ have at least one endpoint in that part (since they all cross $e$), and specifically all edges of $\kappa$ are either in 
		$G^g$ or $G^r$, which implies that $|\kappa| \le 3$ by the induction hypothesis. 
		Hence, we assume that $e$ belongs to
		$E(R_3 \arr R_4, R_3 \arr x, x \arr R_4, G_3 \arr G_4, G_3 \arr x, x \arr G_4)$ and that none of the edges of $\kappa$ contains both endpoints in the same part $R_i$ or $G_i$ of $V$.
		
		\begin{itemize}   
			\item If $e \in E(G_3 \arr x, x \arr G_4)$, then the only edges potentially crossing $e$ are in $E(G_3 \arr G_4, s \arr G_4, s \arr G_3)$, that is, they all belong to $G^g$.
			In that case $|\kappa| \le 3$ by the hypothesis \ref{ss:I2} applied to $G^g$, a contradiction.
			Therefore, $\kappa \cap E(G_3 \arr x, x \arr G_4) = \emptyset$.
			
			\item If $e \in E(G_3 \arr G_4)$, then the edges crossing $e$ are
			either incident to $s$, or incident to $x$, or in $E(G_3 \arr G_4)$. Since $\tw{G_3 \arr G_4} \le 1$, 
			we have that $|\kappa| \le 3$.
			Therefore, $\kappa \cap E(G_3 \arr G_4) = \emptyset$.
			
			\item If $e \in E(R_3 \arr x)$, then the edges crossing $e$ are in
			$E(s \arr G_3, t \arr R_3, R_3 \arr R_4)$.
			Observe that each of the three subsets contributes at most one edge to $\kappa$; thus, $|\kappa| \le 3$.
			Therefore, $\kappa \cap E(R_3 \arr x) = \emptyset$.
			
			\item If $e \in E(x \arr R_4)$, then the edges crossing $e$ are in
			$E(s \arr G_4, t \arr R_4, R_3 \arr R_4)$.
			Each of the three subsets contributes at most one edge to $\kappa$; thus, $|\kappa| \le 3$.
			Therefore, $\kappa \cap E(x \arr R_4) = \emptyset$.
			
			\item If $e \in E(R_3 \arr R_4)$, then the edges crossing $e$ are either adjacent
			to $s$ or $t$, or in $E(R_3 \arr R_4)$.
			Since $\tw{R_3 \arr R_4} \le 1$, we have that $|\kappa| \le 3$, contradicting our assumption.
		\end{itemize}   		
	\end{enumerate}
		
	\paragraph{Case 2.}
	Now assume $(s, x) \in E$ and $(x, t) \in E$; see \cref{fig:ss_odag_case2}.
	Again, $G$ is decomposed into two edge-disjoint single-source
	subgraphs sharing a single vertex $x$; denote the graph containing $(s, x)$ by $G^g$ and the graph containing
	$(x, t)$ by $G^r$, where $s$ is the single source of $G^g$ and $x$ is the single source of $G^r$. By the induction hypothesis, the two graphs admit orders $\sigma^g$ and $\sigma^r$ satisfying the invariant. 	
	Let $\sigma^g = [s, G_3, x, G_4]$ and $\sigma^r = [x, R_3, t, R_4]$. 
	Then 
	$\sigma = [s, G_3, x, G_4, R_3, t, R_4]$, where
	$H_3 = [G_3, x, G_4, R_3]$, $H_4 = R_4$; see \cref{fig:ss_odag_case2_ind}.	
	In \cref{sect:ss_odag_app} we show that the invariants are maintained.
\end{sketch} 

The recent result of Davies~\cite{d-ibccg-22} implies that the stack number of single-source outerplanar DAGs is at most $48$.
We reduce this upper bound on the stack number to $4$ via a similar argument that
employs the same recursive decomposition as in \cref{thm:ss_odag_twist}.
The proof of \cref{thm:ss_odag_stack} is in the appendix. 

\wormholeLm{lm-ss-odag-stack}
\begin{lemma}
	\label{thm:ss_odag_stack}
	Every single-source outerplanar DAG admits a $4$-stack layout.
\end{lemma}

It is straightforward to extend \cref{thm:ss_odag_stack} to oDAGs with a constant number of sources (sinks),
that is, to construct a layout of an oDAG with $4s$ stacks, where $s$ is the number of sources (sinks) in the graph.
Partition the oDAG into $s$ single-source subgraphs and embed each of them in a separate set of $4$ stacks.

\subsection{Monotone oDAGs}
\label{sect:mon_odag}

Here we consider \df{monotone} outerplanar DAGs that are constructed from an edge by applying 
operations \ref{op2} and \ref{op3}. As in the previous section, we assume that the construction sequence
along with the base edge is known.

\wormholeLm{lm-odag-twist}
\begin{lemma}
	\label{thm:odag_twist}
	Every monotone outerplanar DAG admits an order whose twist size is at most $4$.
\end{lemma}

\begin{sketch}
	We prove the claim by induction on the size of the given oDAG, $G=(V, E)$, by using
	the following invariants (see \cref{fig:odag_inv}):
	For a base edge $(s, t) \in E$, there exists a vertex order consisting of six parts,
	$\sigma = [H_1$, $s$, $H_2$, $H_3$, $t$, $H_4]$,
	such that the following holds:
	
	\begin{figure}[!tb]
		\center
		\includegraphics[page=1,height=70pt]{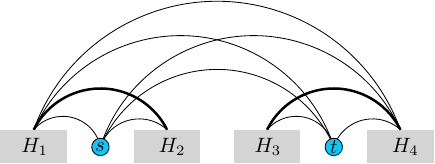}
		\caption{The invariant used in the proof of \cref{thm:odag_twist}}
		\label{fig:odag_inv}
	\end{figure}
	
\smallskip
	\noindent~~\begin{enumerate*}[label=\textbf{I.\arabic*},ref=I.\arabic*,itemjoin=~~]
		\item\label{I0} $E(H_1 \arr H_3) = E(H_2 \arr H_3) = E(H_2 \arr H_4) = \emptyset$ \\
		
		\item\label{I1} $\tw{H_1 \cup \{s\} \arr \{t\} \cup H_4} \le 1$ \\
		
		\item\label{I2} $\tw{H_1 \arr H_2 \cup \{t\} \cup H_4} \le 2$ \qquad\qquad

		\item\label{I3} $\tw{H_1 \cup \{s\} \cup H_3 \arr H_4} \le 2$ \\

		\item\label{I4} $\tw{E} \le 4$
	\end{enumerate*} 
\smallskip
	
	If $G$ consists of a single edge, then the base of the induction clearly holds. For the inductive case, 
	consider the base edge $(s, t)$ of $G$ and choose the unique common neighbor of $s$ and $t$, 
	denoted $x \in V$. Since $G$ is monotone and acyclic, there are
	two ways the edges between $x$ and $s, t$ are directed: either 
	$(s, x) \in E, (t, x) \in E$ (\ref{op2}) or $(x, s) \in E, (x, t) \in E$ (\ref{op3}).
	Observe that, since a (monotone) outerplanar DAG remains (monotone) outerplanar after reversing all 
	edge directions and the described invariants are symmetric with respect to parts $H_1, H_2$ and parts $H_3, H_4$, 
	it is sufficient to study only one of the two cases. Therefore we investigate
	the former case, while the latter case follows from the symmetry.
	
	\begin{figure}[!tb]
		\center
		\includegraphics[page=2,width=\linewidth]{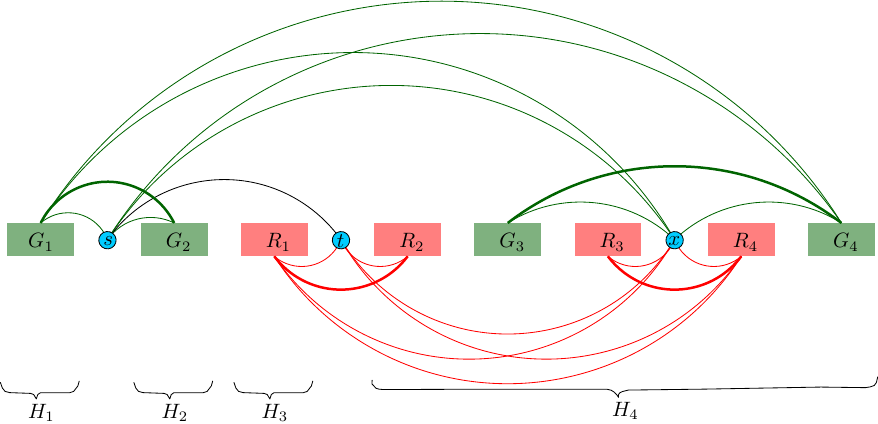}
		\caption{An inductive step used in the proof of \cref{thm:odag_twist}}
		\label{fig:odag_case1}
	\end{figure}
	
	Assume $(s, x) \in E$ and $(t, x) \in E$. It is easy to see that $G$ is decomposed into two edge-disjoint monotone
	oDAGs sharing a vertex $x \in V$; denote the graph containing $(s, x)$ by $G^g$ and the graph containing
	$(t, x)$ by $G^r$. By the induction hypothesis, the two graphs admit orders $\sigma^g$ and $\sigma^r$ satisfying the
	described invariant. 	
	Let $\sigma^g = [G_1, s, G_2, G_3, x, G_4]$ and $\sigma^r = [R_1, t, R_2, R_3, x, R_4]$. Then
\[
	\sigma = [G_1, s, G_2, R_1, t, R_2, G_3, R_3, x, R_4, G_4]
\] 
where $H_1 = G_1$, $H_2 = G_2$, $H_3 = R_1$, $H_4 = [R_2, G_3, R_3, x, R_4, G_4]$ in~$\sigma$; see \cref{fig:odag_case1}.
	In \cref{sect:mon_odag_app} we show that the invariants are maintained under $\sigma$.
\end{sketch}

\subsection{Outerpath DAGs}
\label{sect:outerpath_dag}

Let $G$ be an embedded (plane) graph. Recall that a \df{weak dual}
is a graph whose vertices are bounded faces of $G$ and edges connect adjacent faces of $G$.
A graph is an \df{outerpath} if its weak dual is a path.
Consider a face of an outerpath $G=(V, E)$ that corresponds to a terminal of the path, and 
make an edge on the face adjacent to the outerface of $G$ to be a base edge.
It is easy to see that
every outerpath can be constructed from such a base edge by repeatedly stellating edges
such that the following holds (which keeps the weak dual to be a path):
After stellating edge $(u, v)$ with a vertex $w$, only one of the two newly added edges, $\{u, w\}$ and $\{v, w\}$,
can be further stellated. In order to construct an outerpath DAG, the directions of the edges have
to follow one of the operations, \ref{op1}, \ref{op2}, or \ref{op3}.

\wormholeLm{lm-outerpath-twist}
\begin{lemma}
	\label{thm:outerpath_twist}
	Every outerpath DAG admits an order whose twist size is at most~$4$.
\end{lemma}

\begin{figure}[!bt]
	\begin{subfigure}[b]{.24\linewidth}
		\center
		\includegraphics[page=4,width=\linewidth]{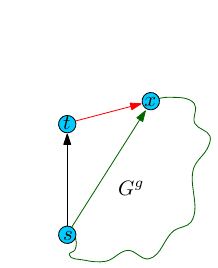}
		\caption{Case~1a}
		\label{fig:outerpath_case1a}
	\end{subfigure}    
	\hfill
	\begin{subfigure}[b]{.24\linewidth}
		\center
		\includegraphics[page=5,width=\linewidth]{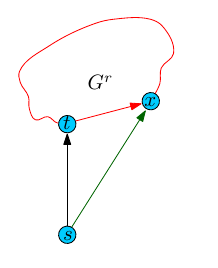}
		\caption{Case~1b}
		\label{fig:outerpath_case1b}
	\end{subfigure}   
	\hfill
	\begin{subfigure}[b]{.24\linewidth}
		\center
		\includegraphics[page=2,width=\linewidth]{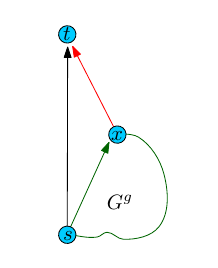}
		\caption{Case~2a}
		\label{fig:outerpath_case2a}
	\end{subfigure}   
	\hfill
	\begin{subfigure}[b]{.24\linewidth}
		\center
		\includegraphics[page=3,width=\linewidth]{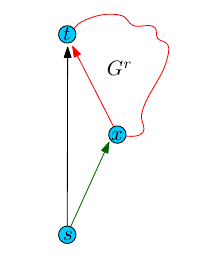}
		\caption{Case~2b}
		\label{fig:outerpath_case2b}
	\end{subfigure}   
	\caption{Cases in \cref{thm:outerpath_twist}: stellating base edge $(s, t)$ with a vertex $x$}
	\label{fig:outerpath_cases}
\end{figure}

\begin{sketch}
	We prove the claim by induction on the size of the given outerpath DAG, $G=(V, E)$, by using
	the following invariants:
	For a base edge $(s, t) \in E$, there exists a vertex order
	consisting of six parts, $\sigma = [H_1$, $s$, $H_2$, $H_3$, $t$, $H_4]$, 
	such that the following holds:
	
\smallskip
	\noindent\begin{enumerate*}[label=\textbf{I.\arabic*},ref=I.\arabic*,itemjoin=]
		\item\label{p:I0} $H_2 = \emptyset$ or $H_3 = \emptyset$, that is, $E(H_2 \arr H_3) = \emptyset$ \\
		
		\item\label{p:I2a} $\tw{H_2 \arr t, H_2 \arr H_4} \le 1$ ~~~~\qquad\qquad
		
		\item\label{p:I2b} $\tw{H_1 \arr H_3, s \arr H_3} \le 1$ \\
		
		\item\label{p:I3} $\tw{H_1 \cup \{s\} \cup H_2 \arr H_3 \cup \{t\} \cup H_4} \le 2$ \\
		
		\item\label{p:I4a} $\tw{H_1 \arr H_2, H_2 \arr \{t\} \cup H_4} \le 3$~~~
		
		\item\label{p:I4b} $\tw{H_1 \cup \{s\} \arr H_3, H_3 \arr H_4} \le 3$ \\
		
		\item\label{p:I5} $\tw{H_1 \cup \{s\} \cup H_2 \cup H_3 \arr H_4} \!\le\! 3$~~

		\item\label{p:I6} $\tw{H_1 \arr H_2 \cup H_3 \cup \{t\} \cup H_4} \!\le\! 3$ \\
		
		\item\label{p:I7} $\tw{E} \le 4$
	\end{enumerate*} 
\smallskip			
			
	If $G$ consists of a single edge, then the base of the induction clearly holds. 
	For the inductive case, consider
	a base edge $(s, t) \in E$ and choose the unique common neighbor of $s$ and $t$, 
	denoted $x \in V$. Although all three operations can be applied on $(s, t)$,
	by symmetry, it is sufficient to study only \ref{op1} and \ref{op2}. Depending on which edges of face
	$\langle s, t, x \rangle$ are utilized for the construction, we distinguish four cases; see \cref{fig:outerpath_cases}.
	As in earlier proofs we denote the graphs constructed on $(s, x)$ by $G^g$ and the graph on $(t, x)$ by $G^r$, and
	assume the graphs admit orders $\sigma^g$ and $\sigma^r$ satisfying the invariants.
	Notice however that, since $G$ is an outerpath, only one of $G^g, G^r$ contains more than two vertices.
	
	\begin{figure}[!t]
		\begin{subfigure}[b]{.48\linewidth}
			\center
			\includegraphics[page=6,width=\linewidth]{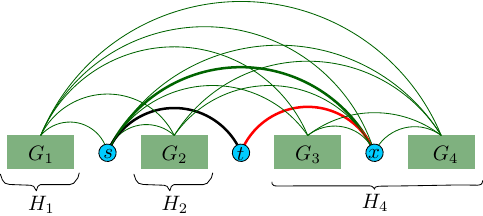}
			\caption{Case~1a}
			\label{fig:outerpath_1a}
		\end{subfigure}    
		\hfill
		\begin{subfigure}[b]{.48\linewidth}
			\center
			\includegraphics[page=7,width=\linewidth]{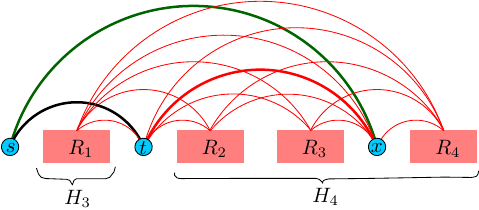}
			\caption{Case~1b}
			\label{fig:outerpath_1b}
		\end{subfigure}   
		
		\begin{subfigure}[b]{.48\linewidth}
			\center
			\includegraphics[page=8,width=\linewidth]{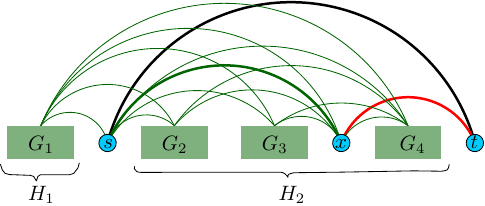}
			\caption{Case~2a}
			\label{fig:outerpath_2a}
		\end{subfigure}    
		\hfill
		\begin{subfigure}[b]{.48\linewidth}
			\center
			\includegraphics[page=9,width=\linewidth]{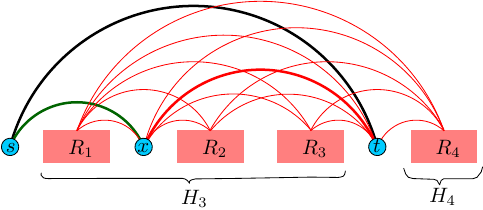}
			\caption{Case~2b}
			\label{fig:outerpath_2b}
		\end{subfigure}   
		\caption{An illustration for \cref{thm:outerpath_twist}}
	\end{figure}
		
	\paragraph{Case 1a.} 	
	Assume that $(s, x) \in E$, $(t, x) \in E$, $\sigma^g = [G_1, s, G_2, G_3, x, G_4]$, and $\sigma^r = [t, x]$.
	We set 
	\(\sigma = [G_1, s, G_2, t, G_3, x, G_4]\), where
	$H_1 = G_1$, $H_2 = G_2$, $H_3 = \emptyset$, $H_4 = [G_3, x, G_4]$ in $\sigma$; see \cref{fig:outerpath_1a}.
		
	\paragraph{Case 1b.} 	
	Assume $(s, x) \in E$, $(t, x) \in E$, $\sigma^g = [s, x]$, and $\sigma^r = [R_1, t, R_2, R_3, x, R_4]$.
	We set 
	\(\sigma = [s, R_1, t, R_2, R_3, x, R_4]\), where
	$H_1 = H_2 = \emptyset$, $H_3 = R_1$, $H_4 = [R_2, R_3, x, R_4]$ in $\sigma$; see \cref{fig:outerpath_1b}.
		
	\paragraph{Case 2a.} 
	Assume $(s, x) \in E$, $(x, t) \in E$, $\sigma^g = [G_1, s, G_2, G_3, x, G_4]$, and $\sigma^r = [x, t]$.
	We set 
	\(\sigma = [G_1, s, G_2, G_3, x, G_4, t]\), where
	$H_1 = G_1$, $H_2 = [G_2, G_3, x, G_4]$, $H_3 = H_4 = \emptyset$ in $\sigma$; see \cref{fig:outerpath_2a}.	
		
	\paragraph{Case 2b.} 
	Assume $(s, x) \in E$, $(x, t) \in E$, $\sigma^g = [s, x]$, and \(\sigma^r = [R_1, x, R_2, R_3, t, R_4]\).
	The case is reduced to \emph{Case~2a} by reversing all edge directions; see \cref{fig:outerpath_2b}.
	
	\noindent\cref{sect:outerpath_dag_app} shows that the invariants are maintained in each of the cases.
\end{sketch}

\section{Upward Planar $3$-Trees}
\label{sect:up3t}

\begin{backInTimeThm}{thm-up3t}
	\begin{theorem}
		Every upward planar $3$-tree admits an order whose twist size is at most $5$.
	\end{theorem} 
\end{backInTimeThm}

\begin{sketch}
	We prove the claim by induction on the size of a given upward planar $3$-tree, $G=(V, E)$, by using
	the following invariants (see \cref{fig:up3t_inv}):
	For the outerface $\langle s, m, t \rangle$ of $G$, there exists a vertex order 
	consisting of five parts, $\sigma = [s, H_1, m, H_2, t]$, where $H_1, H_2 \subset V$, 
	and the following holds:
	
	\smallskip
	\begin{enumerate*}[label=\textbf{I.\arabic*},ref=I.\arabic*]
		\item\label{up3t:I1} $\tw{\{s\} \cup H_1 \arr H_2 \cup \{t\}} \le 2$ \qquad\qquad
		
		\item\label{up3t:I2} $\tw{E} \le 5$
	\end{enumerate*} 
	\smallskip

	The base of the induction clearly holds when $G$ is a triangle.
	For the inductive case, consider the outerface, $\langle s, m, t \rangle$, of $G$
	and identify the unique vertex, $x \in V$, adjacent to $s, m, t$. Since $G$ is upward planar, 
	we have $(s, x) \in E$, $(x, t) \in E$; for the direction of the edge between $x$ and $m$, there are
	two possible cases. 
	We can reduce
	one case to another one by reversing edge directions, which preserves upward
	planarity of the graph. Therefore, we study only one of the cases.

	\begin{figure}[!tb]
		\begin{subfigure}[b]{.55\linewidth}
			\center
			\includegraphics[page=1,height=70pt]{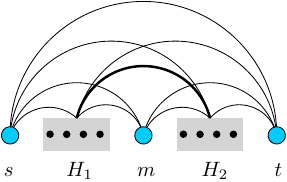}
			\caption{The invariant used in the proof of \cref{thm:up3t}}
			\label{fig:up3t_inv}
		\end{subfigure}    
		\hfill
		\begin{subfigure}[b]{.42\linewidth}
			\center
			\includegraphics[page=2,height=80pt]{pics/up3t}
			\caption{Decomposing an upward planar 3-tree into $G^{r}$, $G^{g}$, and $G^{b}$}
			\label{fig:up3t_dec}
		\end{subfigure} 
		\caption{Bounding the twist size of upward planar 3-trees}
	\end{figure}
	
	Assume $(x, m) \in E$. Then $G$ is decomposed into three upward planar subgraphs
	bounded by faces $\langle s, x, t \rangle$, $\langle s, x, m \rangle$, and $\langle x, m, t \rangle$;
	denote the graphs by $G^{r}$, $G^{g}$, and $G^{b}$, respectively; see \cref{fig:up3t_dec}. 
	By the induction hypothesis, the three graphs admit orders $\sigma^{r}, \sigma^{g}, \sigma^{b}$ satisfying the
	described invariants. 
	Let $\sigma^{r} = [s, R_1, x, R_2, t]$, $\sigma^{g} = [s, G_1, x, G_2, m]$, and $\sigma^{b} = [x, B_1, m, B_2, t]$. 
	Then
\[
	\sigma = [s, R_1, G_1, x, G_2, R_2, B_1, m, B_2, t]
\] 
	where \(H_1 = [R_1, G_1, x, G_2, R_2, B_1]$ and $H_2 = B_2\); see \cref{fig:up3t_case1} in \cref{sect:up3t_app}, where we show that the invariants are maintained under the vertex order.	
\end{sketch}

\section{Lower Bounds}
\label{sect:lower}

We  construct and computationally verify specific graphs that require a minimum number of stacks in every layout
utilizing a SAT formulation of the linear layout problem~\cite{BKKPRU20,Pup20}.
Using a modern SAT solver, one can evaluate small and medium size 
instances (up to a few hundred of vertices) within a few seconds.
An online tool and the source code of the implementation is available at~\cite{bob}.

Using the formulation, we identified a single-source single-sink upward oDAG that requires
three stacks; see \cref{fig:3out}. There are only two linear extensions of the graph, 
$[a, b, c, d, e, f]$ and $[a, b, d, c, e, f]$, and both require three stacks.

Next we found an upward outerplanar DAG with four sources and three sinks whose stack number is $4$; see \cref{fig:4out}.
This oDAG is upward but not monotone, that is, it requires an addition of transitive edges via operation \ref{op1}.

Finally, we construct an upward planar $3$-tree that requires five stacks; see \cref{fig:5st-1}.
The results are summarized in \cref{thm:other}.

\begin{figure}[!tb]
	\begin{subfigure}[b]{.48\linewidth}
		\center
		\includegraphics[page=1,height=100pt]{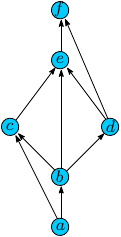}
		\caption{A singles-source single-sink outerplanar DAG that requires three stacks}
		\label{fig:3out}
	\end{subfigure}    
	\hfill
	\begin{subfigure}[b]{.48\linewidth}
		\center
		\includegraphics[page=7,height=160pt]{pics/lower_bounds}
		\caption{An upward outerplanar DAG that requires four stacks}
		\label{fig:4out}
	\end{subfigure}    
	\caption{Lower bound examples}
\end{figure}

\section{Conclusions}

In this paper we studied the stack number of upward planar and outerplanar DAGs and provided improved upper and lower bounds for some interesting subclasses via their maximum twist sizes. With the recent results of Jungeblut et al.~\cite{jmu-daoghcsn-22} one of the intriguing open questions is to decrease the gap between our lower bound of $4$ and their upper bound of $24776$ for oDAGs. Moreover, since our upper bounds are mostly based on bounding the twist number, they are likely too large and it would be interesting to decrease them further. 

A \df{queue layout} of DAGs is a related concept, in which a pair of edges cannot nest. %
While two queues are sufficient for trees and unicyclic DAGs~\cite{HPT99a}, there exist single-source single-sink upward oDAGs
that require a linear number of queues; see \cref{fig:nq}. This is in contrast with undirected planar graphs, which have
a constant queue number~\cite{ABGKP18,DJMMUW19}. 
We suggest to investigate mixed stack-queue layouts in which every page
is either a stack or a queue~\cite{Pup17,CKN19,abkm-mllsg-22}.
Another direction is to parameterize the queue number by a graph parameter that is tied to the queue number 
for undirected graphs,
such as the width of a poset~\cite{ABGKP20,Pup22}.
	
\begin{figure}[!tb]
	\begin{subfigure}[b]{.62\linewidth}
		\center
		\includegraphics[page=8,height=140pt]{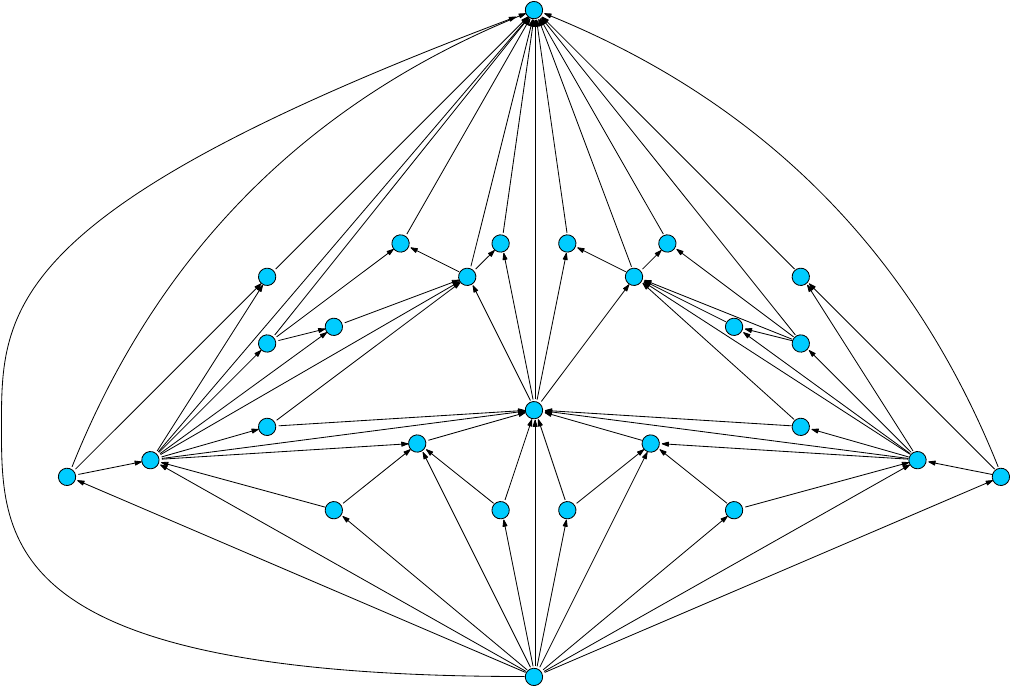}
		\caption{}
		\label{fig:5st-1}
	\end{subfigure}    
	\hfill
	\begin{subfigure}[b]{.34\linewidth}
		\center
		\includegraphics[page=4,height=110pt]{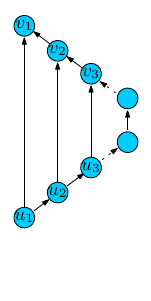}
		\caption{}
		\label{fig:nq}
	\end{subfigure}    
	\caption{(a)~An upward planar DAG that require $5$ stacks. 
		(b)~An upward outerplanar DAG that requires $n/2$ queues.}
	\label{fig:lb2}
\end{figure}

\paragraph*{Acknowledgments.}	
We thank the organizers and other participants of the 2019 Dagstuhl seminar 
``Beyond-Planar Graphs: Combinatorics, Models and Algorithms'', where this work started, 
in particular F.~Frati and T.~Mchedlidze.

\bibliographystyle{splncs04}
\bibliography{refs}
	
\clearpage
\appendix
\chapter*{\appendixname}

\section{Complete Proofs for Section~3}

\subsection{Single-Source oDAGs}
\label{sect:ss_odag_app}

Here we provide a complete proof of \cref{thm:ss_odag_twist}.

\begin{backInTimeLm}{lm-ss-odag-twist}
\begin{lemma}
	Every single-source (single-sink) outerplanar DAG admits an order whose twist size is at most $3$.
\end{lemma}
\end{backInTimeLm}

\begin{proof}
	Let $G = (V, E)$ be a given oDAG with a unique source $s \in V$, and assume that $(s, t) \in E$ is 
	the base edge in the construction sequence of $G$.	
	We prove the claim by induction on the size of $G$ by using
	the following invariant (see \cref{fig:ss_odag_inv}):
	There exists an order of $V$ consisting of four parts, 
	$\sigma = [s , H_3, t, H_4]$ (that is, $H_1 = H_2 = \emptyset$), 
	such that the following holds:
	\begin{enumerate}[label=\textbf{I.\arabic*},ref=I.\arabic*]   
		\item\label{app_ss:I1} $\tw{s \arr H_4, H_3 \arr H_4} \le 1$;
		
		\item\label{app_ss:I2} $\tw{E} \le 3$.
	\end{enumerate} 
	
	Now we prove that the described invariants can be maintained.
	If $G$ consists of a single edge, then the base of the induction 
	clearly holds. In order to prove the inductive case, we consider
	the base edge $(s, t)$ of $G$ and choose the unique neighbor of $s$ and $t$, 
	denoted $x \in V$. Since $G$ is a single-source oDAG, there are
	two ways the edges between $x$ and $s, t$ are directed, corresponding to
	operations \ref{op1} and \ref{op2}. Consider both cases.
	
	\paragraph{\bf Case 1.} 
	First assume $(s, x) \in E$ and $(t, x) \in E$. It is easy to see that $G$ is decomposed into two edge-disjoint 
	subgraphs sharing a single vertex $x$; denote the graph containing $(s, x)$ by $G^g$ and the graph containing
	$(t, x)$ by $G^r$; see \cref{fig:ss_odag_case1}. 
	Since $G$ is a single-source oDAG, $G^g$ is also a single-source oDAG with source $s$ and base edge $(s, x)$.
	Similarly, $G^r$ is a single-source oDAG with source $t$ and base edge $(t, x)$.
	By the induction hypothesis, the two graphs admit orders $\sigma^g$ and $\sigma^r$ satisfying the
	described invariant. Next we show how to combine the orders into a single one for $G$. 
	
	Let $\sigma^g = [s, G_3, x, G_4]$ and $\sigma^r = [t, R_3, x, R_4]$. 
	Then we set 
		\[	
		\sigma = [s, t, R_3, G_3, x, G_4, R_4]
		\]	
	and observe that in the order
	$H_3 = \emptyset$, $H_4 = [R_3, G_3, x, G_4, R_4]$; see \cref{fig:ss_odag_case1_ind}.
	
	It is easy to see that $\sigma$ is a linear extension of $V$, that is, $u <_{\sigma} v$ for all edges $(u, v) \in E$.
	Next we verify the conditions of the invariant.
	\begin{enumerate}
		\item[\ref{app_ss:I1}.] 
		Since $H_3 = \emptyset$, we have 
		$\tw{s \arr H_4, H_3 \arr H_4} = \tw{s \arr H_4} \le 1$, where
		the inequality holds because all edges $E(s \arr H_4)$ share a common vertex, $s$.
		
		\item[\ref{app_ss:I2}.] 
		Consider the maximum twist $\kappa$ in $G$ under vertex order $\sigma = [s, H_3, t, H_4]$, and
		suppose for contradiction that $|\kappa| \ge 4$.
		Observe that $H_3 = \emptyset$ in the considered case, and $\kappa$ may contain at most one edge incident to $s$
		and at most one edge incident to $t$. Thus, at least two of the edges of $\kappa$
		are from $E(H_4 \arr H_4)$; see \cref{fig:ss_odag_case1_ind}.
		Denote one of the two edges by $e \in \kappa$.
		
		Since $e \in E(H_4 \arr H_4)$, both endpoints of $e$ are in $R_3 \cup G_3 \cup \{x\} \cup G_4 \cup R_4$.
		Notice that if the two endpoints are in the same part (e.g., $R_i$ or $G_i$ for some $i$), then all edges of
		$\kappa$ have endpoints in that part (since they all cross $e$), and specifically all edges of $\kappa$ are either in 
		$G^g$ or $G^r$, which implies that $|\kappa| \le 3$ by the induction hypothesis. 
		Hence, we assume that $e$ belongs to
		$E(R_3 \arr R_4, R_3 \arr x, x \arr R_4, G_3 \arr G_4, G_3 \arr x, x \arr G_4)$ and that none of the edges of $\kappa$ contains both endpoints in the same part $R_i$ or $G_i$ of $V$.
		
		\begin{itemize}   
			\item If $e \in E(G_3 \arr x, x \arr G_4)$, then the only edges potentially crossing $e$ are in
			$E(G_3 \arr G_4, s \arr G_4, s \arr G_3)$, that is, they all belong to $G^g$.
			In that case $|\kappa| \le 3$ by the hypothesis \ref{app_ss:I2} applied to $G^g$, a contradiction.
			Therefore, $\kappa \cap E(G_3 \arr x, x \arr G_4) = \emptyset$.
			
			\item If $e \in E(G_3 \arr G_4)$, then the edges crossing $e$ are
			either incident to $s$, or incident to $x$, or in $E(G_3 \arr G_4)$. Since $\tw{G_3 \arr G_4} \le 1$, 
			we have that $|\kappa| \le 3$.
			Therefore, $\kappa \cap E(G_3 \arr G_4) = \emptyset$.
			
			\item If $e \in E(R_3 \arr x)$, then the edges crossing $e$ are in
			$E(s \arr G_3, t \arr R_3, R_3 \arr R_4)$.
			Observe that each of the three subsets contributes at most one edge to $\kappa$; thus, $|\kappa| \le 3$.
			Therefore, $\kappa \cap E(R_3 \arr x) = \emptyset$.
			
			\item If $e \in E(x \arr R_4)$, then the edges crossing $e$ are in
			$E(s \arr G_4, t \arr R_4, R_3 \arr R_4)$.
			Each of the three subsets contributes at most one edge to $\kappa$; thus, $|\kappa| \le 3$.
			Therefore, $\kappa \cap E(x \arr R_4) = \emptyset$.
			
			\item If $e \in E(R_3 \arr R_4)$, then the edges crossing $e$ are either adjacent
			to $s$ or $t$, or in $E(R_3 \arr R_4)$.
			Since $\tw{R_3 \arr R_4} \le 1$, we have that $|\kappa| \le 3$, contradicting our assumption.
		\end{itemize}   
		
		This completes the proof for invariants of Case 1.
	\end{enumerate}

	\paragraph{\bf Case 2.}
	Now assume $(s, x) \in E$ and $(x, t) \in E$; see \cref{fig:ss_odag_case2}.
	Again, $G$ is decomposed into two edge-disjoint single-source
	subgraphs sharing a single vertex $x$; denote the graph containing $(s, x)$ by $G^g$ and the graph containing
	$(x, t)$ by $G^r$. By the induction hypothesis, the two graphs admit orders $\sigma^g$ and $\sigma^r$ satisfying the
	invariant. Next we show how to combine the orders into a single one for $G$.
	
	Let $\sigma^g = [s, G_3, x, G_4]$ and $\sigma^r = [x, R_3, t, R_4]$. 
	Then we set 
		\[	
		\sigma = [s, G_3, x, G_4, R_3, t, R_4]
		\]	
	and observe that in the order
	$H_3 = [G_3, x, G_4, R_3]$, $H_4 = R_4$; see \cref{fig:ss_odag_case2_ind}.	
	As in Case 1, $\sigma$ is a linear extension of $V$, and we verify the conditions of the invariant.
	\begin{enumerate}
		\item[\ref{app_ss:I1}.] 
		Observe that $E(s \arr H_4) = \emptyset$ and 
		$E(H_3 \arr H_4) = E(x \arr R_4, R_3 \arr R_4)$.
		Thus, $\tw{s \arr H_4, H_3 \arr H_4} = \tw{x \arr R_4, R_3 \arr R_4} \le 1$
		where the inequality follows from the induction hypothesis \ref{app_ss:I1} applied to $G^r$.   
		
		\item[\ref{app_ss:I2}.] 
		Consider the maximum twist $\kappa$ in $G$ under vertex order $\sigma = [s, H_3, t, H_4]$,
		and suppose for contradiction that $|\kappa| \ge 4$.
		First we rule out the case when $\kappa$ does not contain an edge from $E(H_3 \arr H_3, H_4 \arr H_4)$
		(that is, when all edges of $\kappa$ are either in $E(H_3 \arr H_4)$, or adjacent to $s$ or $t$).
		In that case, $\kappa$ contains at most one edge adjacent to $s$ and at most one edge adjacent to $t$, and all the remaining 
		edges are from $E(H_3 \arr H_4)$. By \ref{app_ss:I1} it holds that $\tw{H_3 \arr H_4} \le 1$, which
		implies that $|\kappa| \le 3$.
		
		Therefore, we may assume that at least one edge $e \in \kappa$ is 
		from $E(H_3 \arr H_3, H_4 \arr H_4)$.
		If $e \in E(H_4 \arr H_4)$, then all edges of $\kappa$ are incident to a vertex in $R_4$ 
		(since only such edges cross $e$); in particular all edges of $\kappa$ are in $G^r$.
		This is impossible by the induction hypothesis \ref{app_ss:I2} applied to $G^r$.
		Thus we may assume that $e$ belongs to $E(H_3 \arr H_3)$.
		
		Since $e \in E(H_3 \arr H_3)$, both endpoints of $e$ are in $G_3 \cup \{x\} \cup G_4 \cup R_3$.
		Notice that if the two endpoints are in the same part (e.g., $G_3$, $G_4$, or $R_3$), then all edges of
		$\kappa$ have endpoints in that part (since they all cross $e$), and specifically all edges of $\kappa$ are either in 
		$G^g$ or $G^r$, which implies that $|\kappa| \le 3$ by the induction hypothesis. 
		Hence, we assume that $e$ belongs to
		$E(G_3 \arr G_4, G_3 \arr x, x \arr G_4, x \arr R_3)$ and that none 
		of the edges of $\kappa$ are in the same part $R_i$ or $G_i$ of $V$.
		
		\begin{itemize}   
			\item If $e \in E(G_3 \arr x, x \arr G_4)$, then the only edges crossing $e$ are in
			$E(s \arr G_3, s \arr G_4, G_3 \arr G_4)$,
			that is, they all belong to $G^g$.
			In that case $|\kappa| \le 3$ by the hypothesis \ref{app_ss:I2} applied to $G^g$. 
			Hence, $\kappa \cap E(G_3 \arr x, x \arr G_4) = \emptyset$.
			
			\item If $e \in E(G_3 \arr G_4)$, then the edges crossing $e$ are either incident to $s$, or
			incident to $x$, or in $E(G_3 \arr G_4)$. However, $\tw{G_3 \arr G_4} \le 1$, which implies that
			$|\kappa| \le 3$.
			Therefore, $\kappa \cap E(G_3 \arr G_4) = \emptyset$.
			
			\item If $e \in E(x \arr R_3)$, then the edges crossing $e$ are in
			$E(s \arr G_4, G_3 \arr G_4, R_3 \arr t, R_3 \arr R_4)$.
			Each of the four subsets contributes at most one edge to $\kappa$, and edges of
			$E(s \arr G_4, G_3 \arr G_4)$ do not cross edges of $E(R_3 \arr t, R_3 \arr R_4)$,
			contradicting our assumption that $|\kappa| \ge 4$.
		\end{itemize}   
	\end{enumerate}
	
	This completes the proof of \cref{thm:ss_odag_twist}.
\end{proof} 

\begin{backInTimeLm}{lm-ss-odag-stack}
	\begin{lemma}
		Every single-source outerplanar DAG admits a $4$-stack layout.
	\end{lemma}	
\end{backInTimeLm}

\begin{proof}
	The proof follows the same recursive approach as the proof of \cref{thm:ss_odag_twist}; in particular, 
	the vertex order is the same. Next we describe relevant differences.
	
	For a given outerplanar DAG, $G=(V, E)$ with a unique source $s \in V$, we maintain the following invariant.
	For a base edge $(s, t)$ of $G$, there exists a layout in $4$ stacks, 
	$\Sh_1, \Sh_2, \Sh_3, \Sh_4$, with a vertex order 
	consisting of four parts, $\sigma = [s, H_3, t, H_4]$, such that the following holds (see \cref{fig:ss_odag_inv_stack}):
		
	\begin{enumerate}[label=\textbf{I.\arabic*},ref=I.\arabic*]   
		\item\label{ss:S0} $E(s \arr H_3, s \arr H_4) \subseteq \Sh_1$;
		
		\item\label{ss:S1} $E(H_3 \arr t, t \arr H_4) \subseteq \Sh_2$;
		
		\item\label{ss:S2} $E(H_3 \arr H_4) \subseteq \Sh_3$.
	\end{enumerate}
	
	\begin{figure}[!ht]
		\center
		\includegraphics[page=6]{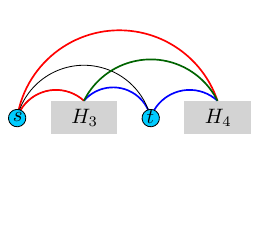}
		\caption{The invariant used in the proof of \cref{thm:ss_odag_stack}}
		\label{fig:ss_odag_inv_stack}
	\end{figure}
	
	Now we show that the invariants can be maintained in the two cases.
	
	\paragraph{Case 1.} 
	Assume $(s, x) \in E$ and $(t, x) \in E$, and that the edges of the two subgraphs are assigned to stacks recursively; 
	see \cref{fig:ss_odag_case1_ind_stack}.
	By renaming stacks, we may assume that
	\begin{itemize}
		\item $E(s \arr G_3, s \arr G_4) \subseteq \Sh_1$ ({\color{red}red}), 
		\item $E(G_3 \arr x, x \arr G_4) \subseteq \Sh_2$ ({\color{myblue}blue}), 
		\item $E(G_3 \arr G_4) \subseteq \Sh_3$ ({\color{green}green}), 
		\item $E(t \arr R_3, t \arr R_4) \subseteq \Sh_2$,
		\item $E(R_3 \arr x, x \arr R_4) \subseteq \Sh_4$ ({\color{mypurple}purple}), 
		\item $E(R_3 \arr R_4) \subseteq \Sh_3$. 
	\end{itemize}
	
	For the remaining three edges we use the following stack assignment:
	$(s, x) \in \Sh_1$, $(t, x) \in \Sh_2$, $(s, t) \in \Sh_2$. One can verify that edges in the same stack
	do not cross each other, and that the invariants, \ref{ss:S0}, \ref{ss:S1}, \ref{ss:S2}, are maintained.
	
	\begin{figure}[!ht]
		\begin{subfigure}[b]{\linewidth}
			\center
			\includegraphics[page=7,height=100pt]{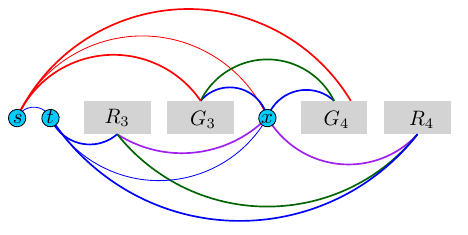}
			\caption{Case~1}
			\label{fig:ss_odag_case1_ind_stack}
		\end{subfigure}   
		\begin{subfigure}[b]{\linewidth}
			\center
			\includegraphics[page=8,height=100pt]{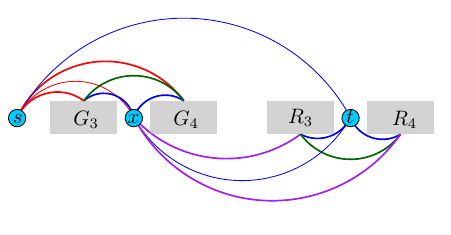}
			\caption{Case~2}
			\label{fig:ss_odag_case2_ind_stack}
		\end{subfigure}   
		\caption{Stack assignment in the proof of \cref{thm:ss_odag_stack}}
	\end{figure}  
	
	\paragraph{Case 2.} 
	Assume $(s, x) \in E$ and $(x, t) \in E$, and that the edges of the two subgraphs are assigned to stacks recursively; 
	see \cref{fig:ss_odag_case2_ind_stack}.
	By renaming stacks, we may assume that
	\begin{itemize}
		\item $E(s \arr G_3, s \arr G_4) \subseteq \Sh_1$ ({\color{red}red}), 
		\item $E(G_3 \arr x, x \arr G_4) \subseteq \Sh_2$ ({\color{myblue}blue}), 
		\item $E(G_3 \arr G_4) \subseteq \Sh_3$ ({\color{green}green}), 
		\item $E(x \arr R_3, x \arr R_4) \subseteq \Sh_4$ ({\color{mypurple}purple}), 
		\item $E(R_3 \arr t, t \arr R_4) \subseteq \Sh_2$,
		\item $E(R_3 \arr R_4) \subseteq \Sh_3$. 
	\end{itemize}
	
	For the remaining three edges we use the following stack assignment:
	$(s, x) \in \Sh_1$, $(x, t) \in \Sh_2$, $(s, t) \in \Sh_2$. One can verify that edges in the same stack
	do not cross each other, and that invariants \ref{ss:S0} and \ref{ss:S1} are maintained.
	In order to show that \ref{ss:S2} is maintained, we make an observation that follows
	directly from the recursive construction:
	\begin{obs}
		For the order $\sigma = [s, H_3, t, H_4]$, we have either
		$E(s \arr H_4) = \emptyset$ or $H_3 = \emptyset$.
	\end{obs} 
	Applying the observation for graph $G^r$, we get that either $E(x \arr R_4) = \emptyset$ or 
	$E(R_3 \arr R_4) = \emptyset$. In both cases we have that edges $E(H_3 \arr H_4)$ are in one stack,
	$\Sh_3$ or $\Sh_4$, implying that \ref{ss:S2} holds.	
	This completes the proof of \cref{thm:ss_odag_stack}.
\end{proof}

\subsection{Monotone oDAGs}
\label{sect:mon_odag_app}

Here we provide a complete proof of \cref{thm:odag_twist}.

\begin{backInTimeLm}{lm-odag-twist}
	\begin{lemma}
		Every monotone outerplanar DAG admits an order whose twist size is at most $4$.
	\end{lemma}
\end{backInTimeLm}

\begin{proof}
	We prove the claim by induction on the size of the given oDAG, $G=(V, E)$, by using
	the following invariants (see \cref{fig:odag_inv}):
	For a base edge $(s, t) \in E$, there exists an order $\sigma$ of $V$ 
	consisting of six parts, $\sigma = [H_1$, $s$, $H_2$, $H_3$, $t$, $H_4]$, 
	such that the following holds:
	\begin{enumerate}[label=\textbf{I.\arabic*},ref=I.\arabic*]   
		\item\label{app_I0} $E(H_1 \arr H_3) = E(H_2 \arr H_3) = E(H_2 \arr H_4) = \emptyset$;
		
		\item\label{app_I1} $\tw{H_1 \cup \{s\} \arr \{t\} \cup H_4} \le 1$;
		
		\item\label{app_I2} $\tw{H_1 \arr H_2 \cup \{t\} \cup H_4} \le 2$;
		
		\item\label{app_I3} $\tw{H_1 \cup \{s\} \cup H_3 \arr H_4} \le 2$.
		
		\item\label{app_I4} $\tw{E} \le 4$;
	\end{enumerate} 
	
	Now we prove that the described invariants can be maintained.
	If $G$ consists of a single edge, then the base of the induction clearly holds. For the inductive case, 
	we consider the base edge $(s, t)$ of $G$ and choose the unique common neighbor of $s$ and $t$, 
	denoted $x \in V$. Since $G$ is monotone and acyclic, there are
	two ways the edges between $x$ and $s, t$ are directed: either 
	$(s, x) \in E, (t, x) \in E$ (operation \ref{op2}) or $(x, s) \in E, (x, t) \in E$ (operation \ref{op3}).
	Observe that, since a (monotone) outerplanar DAG remains (monotone) outerplanar after reversing all 
	edge directions and the described invariants are symmetric with respect to parts $H_1, H_2$ and parts $H_3, H_4$, 
	it is sufficient to study only one of the two cases. Therefore in what follows we investigate
	the former case, while the latter case follows from the symmetry.
	
	Assume $(s, x) \in E$ and $(t, x) \in E$. It is easy to see that $G$ is decomposed into two edge-disjoint monotone
	oDAGs sharing a vertex $x \in V$; denote the graph containing $(s, x)$ by $G^g$ and the graph containing
	$(t, x)$ by $G^r$. By the induction hypothesis, the two graphs admit orders $\sigma^g$ and $\sigma^r$ satisfying the
	described invariant. Next we show how to combine the orders into a single one for $G$.
	
	Let $\sigma^g = [G_1, s, G_2, G_3, x, G_4]$ and $\sigma^r = [R_1, t, R_2, R_3, x, R_4]$. 
	Then we set 
	\[\sigma = [G_1, s, G_2, R_1, t, R_2, G_3, R_3, x, R_4, G_4]\] and note that 
	$H_1 = G_1$, $H_2 = G_2$, $H_3 = R_1$, $H_4 = [R_2, G_3, R_3, x, R_4, G_4]$ in~$\sigma$; see \cref{fig:odag_case1}.
	It is easy to see that $\sigma$ is a linear extension of $V$, and
	we verify the conditions of the invariant.
	
	\begin{enumerate}
		\item[\ref{app_I0}.] The condition follows directly from the construction; see \cref{fig:odag_case1}.
		
		\item[\ref{app_I1}.] 
		Consider $\tw{H_1 \arr H_4, s \arr H_4, H_1 \arr t, s \arr t}$.
		Observe that $E(H_1 \arr t) = E(G_1 \arr t) = \emptyset$, 
		$E(H_1 \arr H_4) = E(G_1 \arr G_4, G_1 \arr x)$, and 
		$E(s \arr H_4) = E(s \arr x, s \arr G_4)$.
		Hence, 
		$\tw{H_1 \arr H_4, s \arr H_4, H_1 \arr t, s \arr t} = 
		\tw{G_1 \arr G_4, s \arr G_4, G_1 \arr x, s \arr x, s \arr t} \le 1,$
		where the inequality follows from the induction hypothesis \ref{app_I1} applied to $G^g$ and the fact
		that edge $(s, t)$ does not cross any other edge of the edge set.
		
		\item[\ref{app_I2}.] 
		Consider $\tw{H_1 \arr H_2, H_1 \arr H_4, H_1 \arr t}$.
		Here $E(H_1 \arr t) = \emptyset$, $E(H_1 \arr H_2) = E(G_1 \arr G_2)$, and
		$E(H_1 \arr H_4) = E(G_1 \arr G_4, G_1 \arr x)$. Therefore,
		\[\tw{H_1 \arr H_2, H_1 \arr H_4, H_1 \arr t} = 
		\tw{G_1 \arr G_2, G_1 \arr G_4, G_1 \arr x} \le 2,
		\]
		where the inequality follows from the hypothesis \ref{app_I2} applied to $G^g$.
		
		\item[\ref{app_I3}.]
		Consider $\tw{H_1 \arr H_4, H_3 \arr H_4, s \arr H_4}$, and let
		$\kappa$ be the maximum twist formed by the edges. Next we argue that $|\kappa| \le 2$.
		
		Observe that $\kappa$ may contain edges from seven subsets: 
		$E(G_1 \arr G_4)$, $E(G_1 \arr x)$, $E(s \arr G_4)$, $E(s \arr x)$, 
		$E(R_1 \arr R_4)$, $E(R_1 \arr x)$, and $E(R_1 \arr R_2)$.
		First consider the case when $\kappa$ contains an edge of $E(R_1 \arr R_2)$. Such an edge
		can only cross edges from $E(R_1 \arr R_4, R_1 \arr x)$. However we have that
		$\tw{R_1 \arr R_2, R_1 \arr R_4, R_1 \arr x} \le 2$ by \ref{app_I2} applied to $G^r$, which implies
		that $|\kappa| \le 2$ in the case. Thus, we may assume that $\kappa$ contains no edge of $E(R_1 \arr R_2)$.
		
		The remaining six edge sets are partitioned into $E(G_1 \arr G_4, G_1 \arr x, s \arr G_4, s \arr x)$ 
		and $E(R_1 \arr R_4, R_1 \arr x)$. The size of the maximum twist in each of the two subsets is
		at most one by \ref{app_I1} applied to $G^g$ and for $G^r$, respectively. Therefore, $|\kappa| \le 2$ as claimed.
		
		\item[\ref{app_I4}.]
		Consider the maximum twist $\kappa$ in $G$ under vertex order $\sigma = [H_1, s, H_2, H_3, t, H_4]$.
		First we rule out the case when $\kappa$ does not contain an edge from $E(H_i \arr H_i)$ for some $1 \le i \le 4$ 
		(that is, when all edges of $\kappa$ are either in $E(H_i \arr H_j)$ for some $i \neq j$, or adjacent to $s$ or $t$).
		In that case, $\kappa$ contains at most one edge adjacent to $s$ and at most one edge adjacent to $t$. The remaining edges
		are from $E(H_1 \arr H_2, H_3 \arr H_4, H_1 \arr H_4)$.
		Observe that on one hand, an edge from $E(H_1 \arr H_2)$ cannot form a twist with an edge from $E(H_3 \arr H_4)$.
		On the other hand, we showed that $\tw{H_1 \arr H_2, H_1 \arr H_4} \le 2$ and 
		$\tw{H_3 \arr H_4, H_1 \arr H_4} \le 2$, which implies that $|\kappa| \le 4$.
		
		Therefore, we may assume that at least one edge of $\kappa$, denoted $e \in \kappa$, is 
		from $E(H_i \arr H_i)$ for some $1 \le i \le 4$.
		If $i = 1$, $i = 2$, or $i = 3$, then all edges of $\kappa$ are adjacent to a vertex from $G_1$, $G_2$, or $R_1$, respectively;
		see \cref{fig:odag_case1}. This is impossible by the induction hypothesis \ref{app_I4} applied to $G^g$ and $G^r$.
		Thus we assume that $e$ belongs to $E(H_4 \arr H_4)$ and that all other edges of $\kappa$ have an endpoint in $H_4$.
		
		Since $e \in E(H_4 \arr H_4)$, both endpoints of $e$ are in $R_2 \cup G_3 \cup R_3 \cup \{x\} \cup R_4 \cup G_4$.
		Notice that if the two endpoints are both in the same part (e.g., $R_i$ or $G_i$ for some $i$), then all edges of
		$\kappa$ have endpoints in that part (since they all cross $e$), and specifically all edges of $\kappa$ are either in 
		$G^g$ or $G^r$, which implies $|\kappa| \le 4$ by the induction hypothesis \ref{app_I4}. Hence, we assume that $e$ belongs to
		$E(G_3 \arr G_4, G_3 \arr x, x \arr G_4, R_3 \arr R_4, R_3 \arr x, x \arr R_4)$ and that none of the edges of $\kappa$ are in the same part $G_i$ or $R_i$ of $V$.
		
		\begin{itemize}   
			\item If $e \in E(R_3 \arr R_4, R_3 \arr x, x \arr R_4)$, then all the edges of $\kappa$ crossing $e$ are
			either adjacent to $x$ or belong to $E(R_1 \arr R_4, R_3 \arr R_4, t \arr R_4)$. 
			Since $\tw{R_1 \arr R_4, R_3 \arr R_4, t \arr R_4} \le 2$ by \ref{app_I3} applied to $G^r$ and
			there is at most one edge in $\kappa$ adjacent to $x$, 
			we have that $|\kappa| \le 3$ in this case.
			Hence, we may assume that $\kappa \cap E(R_3 \arr R_4, R_3 \arr x, x \arr R_4) = \emptyset$.
			
			\item If $e \in E(G_3 \arr G_4, G_3 \arr x, x \arr G_4)$, then all the edges of $\kappa$ crossing $e$ are
			either adjacent to $x$ or belong to 
			$E(G_1 \arr G_4, G_3 \arr G_4, s \arr G_4) \cup E(R_1 \arr R_4, t \arr R_4)$.
			The bound $|\kappa| \le 4$ follows from observations that
			$\tw{G_1 \arr G_4, G_3 \arr G_4, s \arr G_4} \le 2$ (by \ref{app_I3} applied to $G^g$), 
			$\tw{R_1 \arr R_4, t \arr R_4} \le 1$ (by \ref{app_I1} applied to $G^r$), and
			that there is at most one edge in $\kappa$ adjacent to $x$.			
		\end{itemize}
	\end{enumerate} 
	
	This completes the proof of \cref{thm:odag_twist}.
\end{proof} 

\subsection{Outerpath DAGs}
\label{sect:outerpath_dag_app}

Here we provide a complete proof of \cref{thm:outerpath_twist}.

\begin{backInTimeLm}{lm-outerpath-twist}
	\begin{lemma}
		Every outerpath DAG admits an order whose twist size is at most~$4$.
	\end{lemma}
\end{backInTimeLm}

\begin{proof}
	We prove the claim by induction on the size of the given outerpath DAG, $G=(V, E)$, by using
	the following invariants:
	For a base edge $(s, t) \in E$, there exists a vertex order
	consisting of six parts, $\sigma = [H_1$, $s$, $H_2$, $H_3$, $t$, $H_4]$, 
	such that the following holds:
	
	\begin{enumerate}[label=\textbf{I.\arabic*},ref=I.\arabic*]   
		\item\label{app_p:I0} $H_2 = \emptyset$ or $H_3 = \emptyset$, that is, $E(H_2 \arr H_3) = \emptyset$;
		
		\item\label{app_p:I2a} $\tw{H_2 \arr t, H_2 \arr H_4} \le 1$;
		
		\item\label{app_p:I2b} $\tw{H_1 \arr H_3, s \arr H_3} \le 1$;
		
		\item\label{app_p:I3} $\tw{H_1 \cup \{s\} \cup H_2 \arr H_3 \cup \{t\} \cup H_4} \le 2$;
		
		\item\label{app_p:I4a} $\tw{H_1 \arr H_2, H_2 \arr t, H_2 \arr H_4} \le 3$;
		
		\item\label{app_p:I4b} $\tw{H_1 \arr H_3, s \arr H_3, H_3 \arr H_4} \le 3$;
		
		\item\label{app_p:I5} $\tw{H_1 \arr H_4, s \arr H_4, H_2 \arr H_4, H_3 \arr H_4} \le 3$;
		
		\item\label{app_p:I6} $\tw{H_1 \arr H_2, H_1 \arr H_3, H_1 \arr t, H_1 \arr H_4} \le 3$;
		
		\item\label{app_p:I7} $\tw{E} \le 4$.
	\end{enumerate}

	Now we prove that the described invariants can be maintained.
	If $G$ consists of a single edge, then the base of the induction clearly holds.
	In order to prove the inductive case, we consider
	a base edge $(s, t) \in E$ and choose the unique common neighbor of $s$ and $t$, 
	denoted $x \in V$. There are three operations that can be applied on $(s, t)$: \ref{op1}, \ref{op2}, and \ref{op3}.
	By symmetry, it is sufficient to study only \ref{op1} and \ref{op2}. Depending on which edges of face
	$\langle s, t, x \rangle$ are utilized for the construction, we distinguish four cases; see \cref{fig:outerpath_cases}.
	As in earlier proofs we denote the graphs constructed on $(s, x)$ by $G^g$ and the graph on $(t, x)$ by $G^r$, and
	assume the two graphs admit orders $\sigma^g$ and $\sigma^r$ satisfying the invariants.
	Notice however that, since $G$ is an outerpath, only one of $G^g, G^r$ contains more than two vertices.	
		
	\paragraph{\bf Case 1a.} 	
	Assume that $(s, x) \in E$, $(t, x) \in E$, $\sigma^g = [G_1, s, G_2, G_3, x, G_4]$, and $\sigma^r = [t, x]$.
	We set 
	\[\sigma = [G_1, s, G_2, t, G_3, x, G_4],\]  
	where $H_1 = G_1$, $H_2 = G_2$, $H_3 = \emptyset$, $H_4 = [G_3, x, G_4]$ in $\sigma$; see \cref{fig:outerpath_1a}.
	Next we verify the conditions of the invariants.	
	
	\begin{enumerate}
		\item[\ref{app_p:I0}.] Follows directly from the construction.
		
		\item[\ref{app_p:I2a}.] Since $E(H_2 \arr t) = \emptyset$, the claim follows from \ref{app_p:I0} and \ref{app_p:I2a} applied to $G^g$.
		
		\item[\ref{app_p:I2b}.] The claim holds since $H_3 = \emptyset$.
		
		\item[\ref{app_p:I3}.] Consider the maximum twist, $\kappa$, among the specified edges. 
		If $(s, t) \in \kappa$, then the only edges crossing $(s, t)$ are in $E(G_2 \arr G_4, G_2 \arr x)$. 
		By \ref{app_p:I2a} applied to $G^g$, we get the claim.		
		If $(s, t) \notin \kappa$, then all the edges are from the hypothesis \ref{app_p:I3} for $G^g$.
		
		\item[\ref{app_p:I4a}.] Since $E(H_1 \arr H_2, H_2 \arr H_4, H_2 \arr t) = E(G_1 \arr G_2, G_2 \arr G_4, G_2 \arr x)$, 
		the claim follows from the hypothesis \ref{app_p:I4a} applied to $G^g$.
		
		\item[\ref{app_p:I4b}.] The claim holds since $H_3 = \emptyset$.
		
		\item[\ref{app_p:I5}.] The claim follows from \ref{app_p:I3} applied to $G^g$.
		
		\item[\ref{app_p:I6}.] Since $E(H_1 \arr H_2, H_1 \arr H_3, H_1 \arr t, H_1 \arr H_4) = 
		E(G_1 \arr G_2, G_1 \arr G_3, G_1 \arr x, G_1 \arr G_4)$, 
		the claim follows from the hypothesis \ref{app_p:I6} applied to $G^g$.
		
		\item[\ref{app_p:I7}.] Consider the maximum twist, $\kappa$, formed by $E$ under $\sigma$. 
		We may assume that $\kappa$ contains $(s, t)$ or $(t, x)$, as otherwise all edges of $\kappa$ are from $G^g$;
		since the two options are symmetric, we assume $(s, t) \in \kappa$. 
		The only edges crossing $(s, t)$ are in $E(G_1 \arr G_2, G_2 \arr G_4, G_2 \arr x)$.
		By \ref{app_p:I4a}, we have $\tw{G_1 \arr G_2, G_2 \arr G_4, G_2 \arr x} \le 3$, which implies that 
		$|\kappa| \le 4$.
	\end{enumerate}

	\paragraph{\bf Case 1b.} 	
	Assume $(s, x) \in E$, $(t, x) \in E$, $\sigma^g = [s, x]$, and $\sigma^r = [R_1, t, R_2, R_3, x, R_4]$.
	We set 
	\[\sigma = [s, R_1, t, R_2, R_3, x, R_4],\] where
	$H_1 = H_2 = \emptyset$, $H_3 = R_1$, $H_4 = [R_2, R_3, x, R_4]$ in $\sigma$; see \cref{fig:outerpath_1b}.
	Next we verify the conditions of the invariant.	
	
	\begin{enumerate}
		\item[\ref{app_p:I0}.] Follows directly from the construction.
		
		\item[\ref{app_p:I2a}.] The claim holds since $H_2 = \emptyset$.
		
		\item[\ref{app_p:I2b}.] The claim holds since $E(H_1 \arr H_3, s \arr H_3) = \emptyset$.
		
		\item[\ref{app_p:I3}.] The claim holds since the relevant set of edges consists of $(s, t)$ and $(s, x)$.
		
		\item[\ref{app_p:I4a}.] The claim holds since $H_2 = \emptyset$.
		
		\item[\ref{app_p:I4b}.] As $E(H_1 \arr H_3, H_3 \arr H_4, s \arr H_3) = 
		E(R_1 \arr R_2, R_1 \arr R_3, R_1 \arr x, R_1 \arr R_4)$, 
		the claim follows from the hypothesis \ref{app_p:I6} applied to $G^r$.
		
		\item[\ref{app_p:I5}.] Observe that $E(H_1 \arr H_4, s \arr H_4, H_2 \arr H_4, H_3 \arr H_4) = 
		E(s \arr x, R_1 \arr R_2, R_1 \arr R_3, R_1 \arr x, R_1 \arr R_4)$. 
		If $(s, x)$ is in the maximum twist, then it can cross only $E(R_1 \arr R_4)$; by \ref{app_p:I3} applied to $G^r$, 
		we get the desired bound. 
		Otherwise, if $(s, x)$ is not in the maximum twist, then the claim follows from \ref{app_p:I6} applied to $G^r$.
		
		\item[\ref{app_p:I6}.] The claim holds since $H_1 = \emptyset$.
		
		\item[\ref{app_p:I7}.] Consider the maximum twist, $\kappa$, formed by $E$ under $\sigma$. 
		We may assume that $\kappa$ contains $(s, t)$ or $(s, x)$, as otherwise all edges of $\kappa$ are from $G^r$.
		
		If $(s, x) \in \kappa$, then the edges crossing $(s, x)$ are in 
		$E(R_1 \arr R_4, t \arr R_4, R_2 \arr R_4, R_3 \arr R_4)$. By \ref{app_p:I5} we get the bound.
		
		If $(s, t) \in \kappa$, then the edges crossing $(s, x)$ are in 
		$E(R_1 \arr R_4, R_1 \arr x, R_1 \arr R_2, R_1 \arr R_3)$. By \ref{app_p:I6} we get the bound.
	\end{enumerate}
	
	\paragraph{\bf Case 2a.} 
	Assume that $(s, x) \in E$, $(x, t) \in E$, $\sigma^g = [G_1, s, G_2, G_3, x, G_4]$, and $\sigma^r = [x, t]$.
	We set 
	\[\sigma = [G_1, s, G_2, G_3, x, G_4, t],\] where
	$H_1 = G_1$, $H_2 = [G_2, G_3, x, G_4]$, $H_3 = H_4 = \emptyset$ in $\sigma$; see \cref{fig:outerpath_2a}.	
	Next we verify the conditions of the invariant.	
	
	\begin{enumerate}
		\item[\ref{app_p:I0}.] Follows directly from the construction.
		
		\item[\ref{app_p:I2a}.] The claim holds since $H_4 = \emptyset$ and $E(H_2 \arr t) = E(x \arr t)$.
		
		\item[\ref{app_p:I2b}.] The claim holds since $H_3 = \emptyset$.
		
		\item[\ref{app_p:I3}.] The claim holds since the relevant set of edges consists of $(s, t)$ and $(x, t)$.
		
		\item[\ref{app_p:I4a}.] 
		Observe that $E(H_1 \arr H_2, H_2 \arr H_4, H_2 \arr t) = 
		E(G_1 \arr G_2, G_1 \arr G_3, G_1 \arr x, G_1 \arr G_4, x \arr t)$. 
		If $(x, t)$ is in the maximum twist, then it can cross only $E(G_1 \arr G_4)$; by \ref{app_p:I3} applied to $G^g$, 
		we get the bound. 
		Otherwise, if $(x, t)$ is not in the maximum twist, then the claim follows from \ref{app_p:I6} applied to $G^g$.
		
		\item[\ref{app_p:I4b}.] The claim holds since $H_3 = \emptyset$.
		
		\item[\ref{app_p:I5}.] The claim holds since $H_4 = \emptyset$.
		
		\item[\ref{app_p:I6}.]
		As $E(H_1 \arr H_2, H_1 \arr H_3, H_1 \arr t, H_1 \arr H_4) = 
		E(G_1 \arr G_2, G_1 \arr G_3, G_1 \arr x, G_1 \arr G_4)$, 
		the claim follows from the hypothesis \ref{app_p:I6} applied to $G^g$.
		
		\item[\ref{app_p:I7}.] Reduced to \ref{app_p:I7} in Case~1b by relabeling vertices and reversing edge directions.
	\end{enumerate}
	
	\paragraph{\bf Case 2b.} 
	Assume $(s, x) \in E$, $(x, t) \in E$, $\sigma^g = [s, x]$, and $\sigma^r = [R_1, x, R_2, R_3, t, R_4]$.
	The case is reduced to \emph{Case~2a} by reversing all edge directions; see \cref{fig:outerpath_2b}.
	
	This completes the proof of \cref{thm:outerpath_twist}.
\end{proof} 

\section{Complete Proofs for Section~4}
\label{sect:up3t_app}

Here we provide a complete proof of \cref{thm:up3t}.

\begin{backInTimeThm}{thm-up3t}
	\begin{theorem}
		Every upward planar $3$-tree admits an order whose twist size is at most $5$.
	\end{theorem} 
\end{backInTimeThm}

\begin{proof}
	We prove the claim by induction on the size of a given upward planar $3$-tree, $G=(V, E)$, by using
	the following invariants (see \cref{fig:up3t_inv}):
	For the outerface $\langle s, m, t \rangle$ of $G$, there exists an order of $V$ 
	such that it consists of five parts, $\sigma = [s, H_1, m, H_2, t]$, where $H_1, H_2 \subset V$, 
	and the following holds:
	\begin{enumerate}[label=\textbf{I.\arabic*},ref=I.\arabic*]
		\item\label{app_up3t:I1} $\tw{H_1 \arr H_2, s \arr H_2, H_1 \arr t, s \arr t} \le 2$;
		
		\item\label{app_up3t:I2} $\tw{E} \le 5$.
	\end{enumerate} 
	
	Now we prove that the described invariants can be maintained.
	If $G$ is a triangle, then the base of the induction clearly holds. In order 
	to prove the inductive case, we consider the outerface, $\langle s, m, t \rangle$, of $G$
	and identify the unique vertex, $x \in V$, adjacent to $s, m, t$. Since $G$ is upward planar, 
	we have $(s, x) \in E$, $(x, t) \in E$; for the direction of the edge between $x$ and $m$, there are
	two possible options. Observe that as in the proof of \cref{thm:odag_twist}, we can reduce
	one option to another one by reversing the directions of all edges of $G$, which preserves upward
	planarity of the graph. Therefore, it is sufficient to study only one of the options.
	
	\begin{figure}[!ht]
		\center
		\includegraphics[page=3,height=180pt]{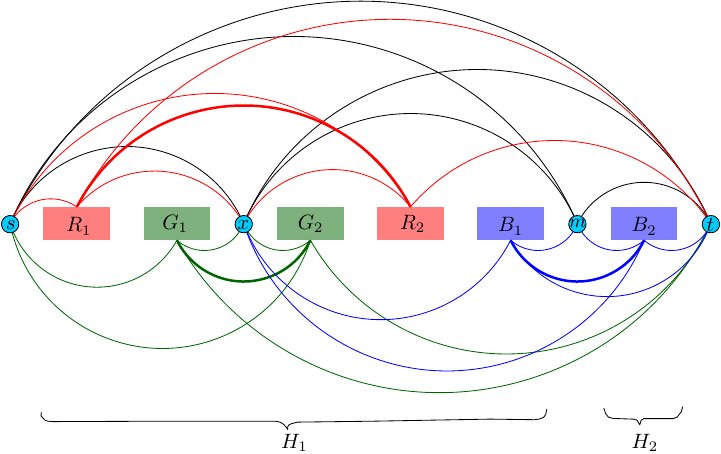}
		\caption{An inductive step in the proof of \cref{thm:up3t}}
		\label{fig:up3t_case1}
	\end{figure}

	Assume $(x, m) \in E$. It is easy to see that $G$ is decomposed into three upward planar subgraphs
	bounded by faces $\langle s, x, t \rangle$, $\langle s, x, m \rangle$, and $\langle x, m, t \rangle$;
	denote the graphs by $G^{r}$, $G^{g}$, and $G^{b}$, respectively; see \cref{fig:up3t_dec}. 
	By the induction hypothesis, the three graphs admit orders $\sigma^{r}, \sigma^{g}, \sigma^{b}$ satisfying the
	described invariants. Next we show how to combine the orders into a single one for $G$.
	
	Let $\sigma^{r} = [s, R_1, x, R_2, t]$, $\sigma^{g} = [s, G_1, x, G_2, m]$, and $\sigma^{b} = [x, B_1, m, B_2, t]$. 
	Then we set 
		\[
		\sigma = [s, R_1, G_1, x, G_2, R_2, B_1, m, B_2, t]
		\] 
	where
	$H_1 = [R_1, G_1, x, G_2, R_2, B_1]$ and $H_2 = B_2$; see \cref{fig:up3t_case1}.
	It is easy to see that $\sigma$ is a linear extension of $V$, and
	we verify the invariant.
	
	\begin{enumerate}
		\item[\ref{app_up3t:I1}]
		Consider edges $E(H_1 \arr H_2, s \arr H_2, H_1 \arr t, s \arr t)$
		and denote the maximum twist formed by these edges by $\kappa$; see \cref{fig:up3t_case1_inv}.
		We need to show that $|\kappa| \le 2$.
		Observe that in this case, $E(s \arr H_2) = \emptyset$, while
		$\tw{\{x\} \cup B_1 \arr B_2 \cup \{t\}} \le 2$ by \ref{app_up3t:I1} applied to $G^b$.
		
		\begin{figure}[!ht]
			\center
			\includegraphics[page=4,height=120pt]{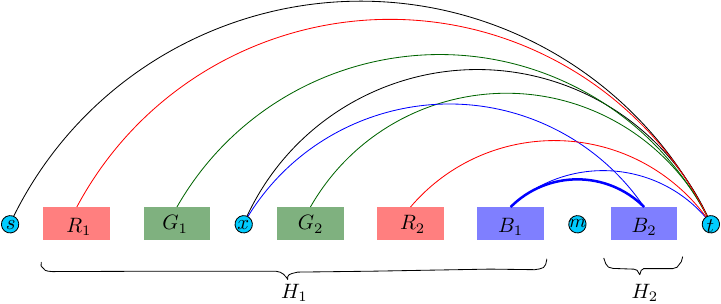}
			\caption{Maintaining invariant \ref{app_up3t:I1} in the proof of \cref{thm:up3t}}
			\label{fig:up3t_case1_inv}
		\end{figure}
		
		Consider edges of $\kappa$ that are adjacent to $t$; clearly, there is at most one such edge. 
		\begin{itemize}
			\item If none of the edges from $E(v \arr t), v \in V$ is in $\kappa$, then $\kappa$ is formed by
			$E(B_1 \arr B_2, x \arr B_2)$; by induction hypothesis \ref{app_up3t:I1} applied to $G^b$, $|\kappa| \le 2$.
			
			\item If an edge, $e \in E(B_1 \arr t)$ is in $\kappa$, then the edges crossing $e$ are 
			from $E(B_1 \arr B_2, x \arr B_2)$. By induction hypothesis applied to $G^b$, $|\kappa| \le 2$.
			
			\item Finally, if an edge $e \in E(s \arr t, R_1 \arr t, x \arr t,  R_2 \arr t)$ 
			is in $\kappa$, then all 
			the edges of $\kappa$ crossing $e$ are from $E(x \arr B_2)$, that is, they are adjacent to $x$. 
			Since only one edge adjacent to a vertex
			can be in a twist, we have $|\kappa| \le 2$.
		\end{itemize}
		
		\item[\ref{app_up3t:I2}]
		Consider the maximum twist $\kappa$ in $G$ under vertex order $\sigma = [s, H_1, m, H_2, t]$.   
		First we rule out the case when $\kappa$ does not contain an edge from $E(H_1 \arr H_1, H_2 \arr H_2)$.
		In that case, $\kappa$ contains at most three edges adjacent to $s, m, t$.
		By \ref{app_up3t:I1} we have $\tw{H_1 \arr H_2} \le 2$, which implies that $|\kappa| \le 5$ in the considered case.
		
		Therefore, we may assume that at least one edge of $\kappa$, denoted $e \in \kappa$, is 
		from $E(H_1 \arr H_1, H_2 \arr H_2)$.
		If $e \in E(H_2 \arr H_2)$, then all edges of $\kappa$ are adjacent to a vertex in $H_2$; by the induction hypothesis
		\ref{app_up3t:I2} applied to $G^b$, we have $|\kappa| \le 5$.
		Thus we may assume that $e$ belongs to $E(H_1 \arr H_1)$ and that all other edges of $\kappa$ have an endpoint in $H_1$, as they cross $e$.
		
		Since $e \in E(H_1 \arr H_1)$, both endpoints of $e$ are in $R_1 \cup G_1 \cup \{x\} \cup G_2 \cup R_2 \cup B_1$.   
		Notice that if the two endpoints are both in the same part (e.g., $R_i$, $G_i$, or $B_i$ for some $i$), then all edges of
		$\kappa$ have endpoints in that part (since they all cross $e$), and specifically all edges of $\kappa$ are either in 
		$G^g$ or $G^r$ or $G^b$, which implies that $|\kappa| \le 5$. Hence, we may assume that $e$ belongs to
		$E(R_1 \arr R_2, R_1 \arr x, x \arr R_2, G_1 \arr G_2, G_1 \arr x, x \arr G_2, x \arr B_1)$ and that none of the 
		edges of $\kappa$ are in the same part $G_i$, $R_i$, or $B_i$ of $V$.
		
		\begin{itemize}   
			\item If $e \in E(G_1 \arr G_2, G_1 \arr x, x \arr G_2)$, then all the edges of $\kappa$ crossing $e$ are either
			adjacent to $s$, $x$, $t$, or belong to $E(G_1 \arr G_2)$. By \ref{app_up3t:I1} applied to $G^g$, we have 
			$\tw{G_1 \arr G_2} \le 2$, which implies that $|\kappa| \le 5$.
			Hence, we may assume that $\kappa \cap E(G_1 \arr G_2, G_1 \arr x, x \arr G_2) = \emptyset$ if $\kappa > 5$.
			
			\item Similarly, if $e \in E(R_1 \arr R_2, R_1 \arr x, x \arr R_2)$, then all the edges of $\kappa$ crossing
			$e$ are either adjacent to $s$, $x$, $t$, or belong to $E(R_1 \arr R_2)$.
			By \ref{app_up3t:I1} applied to $G^r$, we have $\tw{R_1 \arr R_2} \le 2$, which implies that $|\kappa| \le 5$.
			Hence, we may assume that $\kappa \cap E(R_1 \arr R_2, R_1 \arr x, x \arr R_2) = \emptyset$.
			
			\item Finally, if $e \in E(x \arr B_1)$, then all the edges of $\kappa$ crossing $e$ are either adjacent
			to $s$, $m$, $t$, or belong to $E(B_1 \arr B_2)$.
			By \ref{app_up3t:I1} applied to $G^b$, we have $\tw{B_1 \arr B_2} \le 2$.
			At the same time, edges of $\kappa$ adjacent to $s$ (that is, $E(s \arr G_2, s \arr R_2)$) do not
			cross edges from $E(B_1 \arr B_2)$. If $\kappa \cap E(B_1 \arr B_2) = \emptyset$, then 
			$\kappa$ contains at most four edges (adjacent to $s$, $m$, $t$, and $x$). Otherwise, if 
			$\kappa \cap E(B_1 \arr B_2) \neq \emptyset$, $\kappa$ contains
			one edge adjacent to $x$, at most one edge adjacent to $t$, at most one edge adjacent to $m$, 
			and at most two edges from $E(B_1 \arr B_2)$. Therefore, $|\kappa| \le 5$.
		\end{itemize}
	\end{enumerate} 
	
	This completes the proof of \cref{thm:up3t}.  
\end{proof} 
	
\end{document}